\documentclass[11pt,a4paper]{article}
\pdfoutput=1

\usepackage{lmodern}
\usepackage[normalem]{ulem}
\usepackage{titlesec}
\titleformat{\section}
  {\normalfont\fontsize{16pt}{10pt}\selectfont\scshape}
  {\thesection}
  {0.6em}
  {}
\titleformat{\subsection}
  {\normalfont\fontsize{14pt}{8pt}\selectfont\scshape}
  {\thesubsection}
  {0.6em}
  {}

\usepackage{abraces}
\usepackage[utf8]{inputenc}
\usepackage{amssymb}
\usepackage{amsmath}
\usepackage{amsfonts}
\usepackage{bbm}
\usepackage{amsthm}
\usepackage{mathrsfs}
\usepackage{hyperref}
\usepackage{color}
\usepackage[margin=2.5cm]{geometry}
\usepackage[all,cmtip]{xy}
\usepackage{graphicx}
\usepackage{float}
\usepackage{varwidth}
\usepackage{comment}
\usepackage{enumitem}
\usepackage{cancel}

\usepackage{bm}
\usepackage{mathtools}

\usepackage{svg}

\usepackage{upgreek}
\usepackage{rotating}

\usepackage{tikz}
\usetikzlibrary{shapes.geometric}

\usepackage{tikz}
\usetikzlibrary{cd,nfold,matrix,arrows,calc,decorations.pathmorphing,fit,shapes.geometric,decorations.pathreplacing,positioning}
\tikzset{Rightarrow/.style={double equal sign distance,>={Implies},->},
triple/.style={-,preaction={draw,Rightarrow}},
quadruple/.style={preaction={draw,Rightarrow,shorten >=0pt},shorten >=1pt,-,double,double
distance=0.2pt}}

\usepackage{quiver}

\mathchardef\mhyphen="2D

\definecolor{darkred}{rgb}{0.8,0.1,0.1}
\hypersetup{
	colorlinks=true,         % false: boxed links; true: colored links,false is default
	linkcolor=darkred,
	citecolor=blue,
}

\newtheoremstyle{scstyle}%
  {0.4em}   % Space above
  {0.4em}   % Space below
  {}      % Body font
  {}      % Indent amount
  {\fontsize{12pt}{8pt}\selectfont\scshape} % Head font (THIS is the key)
  {.}     % Punctuation after head
  {.5em}  % Space after head
  {}      % Head spec

\theoremstyle{scstyle}
\newtheorem{theo}{Theorem}[section]
\newtheorem{lem}[theo]{Lemma}
\newtheorem{propo}[theo]{Proposition}
\newtheorem{conj}[theo]{Conjecture}
\newtheorem{cor}[theo]{Corollary}

\theoremstyle{scstyle}
\newtheorem{defi}[theo]{Definition}

\newtheorem{exx}[theo]{Example}
\newenvironment{ex}
{\pushQED{\qed}\exx}
{\popQED\endexx}

\newtheorem{remm}[theo]{Remark}
\newenvironment{rem}
{\pushQED{\qed}\remm}
{\popQED\endremm}

\numberwithin{equation}{section}

\def\nn{\nonumber}

\def\v{\varepsilon}

\def\i{\mathrm{I}}
\def\ii{\mathrm{II}}
\def\iii{\mathrm{III}}
\def\iv{\mathrm{IV}}
\def\V{\mathrm{V}}

\def\j{\tilde{j}}
\def\k{\tilde{k}}

\def\p{\tilde{p}}
\def\q{\tilde{q}}

\def\t{\overline{t_{13}}}

\def\bbR{\mathbb{R}}
\def\bbC{\mathbb{C}}
\def\bbN{\mathbb{N}}

\def\KZ{\mathrm{KZ}}
\def\CM{\mathrm{CM}}

\def\End{\mathrm{End}}

\def\ps{\mathrm{ps}}
\def\id{\mathrm{id}}
\def\Id{\mathrm{Id}}

\def\Ch{\mathsf{Ch}}

\def\BB{\mathsf{B}}
\def\C{\mathsf{C}}
\def\D{\mathsf{D}}

\def\E{\mathsf{E}}

\def\r{\mathrm{r}}

\def\dgCat{\mathsf{dgCat}}

\def\ad{\mathrm{ad}}

\def\A{\mathcal{A}}
\def\B{\mathcal{B}}

\def\H{\mathcal{H}}

\def\L{\mathcal{L}}

\def\O{\mathcal{O}}

\def\R{\mathcal{R}}

\newcommand\numberthis{\addtocounter{equation}{1}\tag{\theequation}}

\DeclareMathOperator*{\smallbox}{\text{\raisebox{0.15ex}{\scalebox{0.7}{$\boxtimes$}}}}

\makeatletter
\newcommand{\xRrightarrow}[2][]{\ext@arrow 0359\Rrightarrowfill@{#1}{#2}}
\newcommand{\Rrightarrowfill@}{\arrowfill@\equiv\equiv\Rrightarrow}
\newcommand{\xLleftarrow}[2][]{\ext@arrow 3095\Lleftarrowfill@{#1}{#2}}
\newcommand{\Lleftarrowfill@}{\arrowfill@\Lleftarrow\equiv\equiv}
\makeatother

\def\sk{\vspace{2mm}}

\makeatletter
\let\@fnsymbol\@alph
\makeatother

\title{%
Cartier integration of infinitesimal 2-braidings via 2-holonomy of the CMKZ 2-connection, II: The pentagonator
}

\author{%
Cameron Kemp\vspace{4mm}\\
{\small School of Mathematical Sciences, University of Nottingham,}\\
{\small University Park, Nottingham NG7 2RD, United Kingdom.}\vspace{4mm}\\
{\small \begin{tabular}{ll}
Email: & \href{mailto:cameron.kemp@nottingham.ac.uk}{\texttt{cameron.kemp@nottingham.ac.uk}}
\vspace{2mm}
\end{tabular}
}
}

\begin{document}

\maketitle

\begin{abstract}
\noindent This is a continuation of the previous paper (arXiv:2508.01944) in this series. We recontextualise Cirio and Martins' work to motivate our fundamental conjecture that the Drinfeld-Kohno (Lie) 2-algebra has trivial cohomology. It is then shown that this conjecture implies the following: given a coherent totally symmetric infinitesimal 2-braiding $t$, every modification endomorphic on the zero transformation vanishes if it is made up of the four-term relationators and whiskerings by $t$. The power of such an implication is that, in our context, one need only construct the data of a braided monoidal 2-category and it will automatically satisfy the axioms. We thus conclude by constructing the pentagonator via Cirio and Martins' Knizhnik-Zamolodchikov 2-connection over the configuration space of 4 distinguishable particles on the complex line, $Y_4$. In particular, we make use of Bordemann, Rivezzi and Weigel's pentagon in $Y_4$. 
\end{abstract}



\tableofcontents

\section*{Our 2-categorical quantisation problem}
In order to understand the context of Section \ref{subsec:conjecture} and make this paper relatively self-contained, we must first provide a summary of the relevant material in \cite{Kem25a} regarding infinitesimal 2-braidings and braided monoidal 2-categories. 
\sk

Let us begin by recalling that the category $\Ch^{[-1,0]}$ of cochain complexes concentrated in degrees $\{-1,0\}$ is symmetric monoidal with the monoidal product given by the truncated tensor product $\boxtimes$ and the symmetric braiding given by the swap $\tau$. One can then use this category as a base for enrichment (as in \cite[Chapter 3]{Riehl} or \cite{Kelly}) and study the 2-category $\dgCat^{[-1,0]}$ of $\Ch^{[-1,0]}$-categories, $\Ch^{[-1,0]}$-functors and $\Ch^{[-1,0]}$-natural transformations. As always in enriched category theory, the 2-category itself $\dgCat^{[-1,0]}$ is symmetric monoidal with the monoidal product given by the local truncated tensor product $\smallbox$. To be clear, given a pair of $\Ch^{[-1,0]}$-categories $\C$ and $\D$, we define $\C\smallbox\D$ as having objects given by juxtapositions $UV$ where $U\in\C$ and $V\in\D$, and morphisms given by truncations $f,g:=f\boxtimes g$ where $f\in\C[U,U']$ and $g\in\D[V,V']$. This symmetric monoidal 2-category $\dgCat^{[-1,0]}$ allows one to produce a simple definition of a symmetric strict monoidal $\Ch^{[-1,0]}$-category $(\C,\otimes,I,\gamma)$. The relevant infinitesimal deformations of such a symmetric strict monoidal structure are a weakened variant of $\Ch^{[-1,0]}$-natural transformations.
\begin{defi}\label{def:pseudonatural}
Given $\Ch^{[-1,0]}$-categories $\C,\D$ and $\Ch^{[-1,0]}$-functors $F,G:\C\to\D$, a \textbf{pseudonatural
transformation} $\xi:F\Rightarrow G:\C\to\D$ consists of the following two pieces of data:
\begin{enumerate}
\item[(i)] For each object $U\in \C$, a degree 0 morphism $\xi_U\in\D[F(U),G(U)]^0$.
\item[(ii)] For each pair of objects $U,U'\in \C$, a homotopy $\xi_{(\cdot)}:\C[U,U']\to\D\left[F(U),G(U')\right][-1]$. 
\end{enumerate}
These two pieces of data are required to satisfy the following two axioms: for all $f\in\C[U,U']$ and $f'\in \C[U',U'']$,
\begin{subequations}
\begin{alignat}{2}
G(f)\,\xi_U - \xi_{U'}\,F(f)=&~\partial(\xi_f)+\xi_{\partial(f)}&&,\label{eq:dubindex is homotopy}\\
\xi_{f' f}=&~\xi_{f'}\,F(f) + G(f')\,\xi_f\quad&&.\label{eqn:dubindex splits prods}
\end{alignat}
\end{subequations}
\end{defi}
\begin{defi}\label{def:infinitesimal 2-braiding}
Given a symmetric strict monoidal $\Ch^{[-1,0]}$-category $(\C,\otimes,I,\gamma)$, we say a pseudonatural transformation $t:\otimes\Rightarrow\otimes:\C\smallbox\C\to\C$ is an \textbf{infinitesimal 2-braiding} if, for $f\in\C[U,U']$, $g\in\C[V,V']$ and $h\in\C[W,W']$, we have:
\begin{subequations}\label{subeq:left inf hex}
\begin{alignat}{6}
t_{U(VW)}=&~t_{UV}\otimes1_W+(\gamma_{VU}\otimes1_W)(1_V\otimes t_{UW})(\gamma_{UV}\otimes1_W)\quad&&,\label{eq:t_U(VW)}\\
t_{f,g\otimes h}=&~t_{f,g}\otimes h+(\gamma_{V'U'}\otimes1_{W'})(g\otimes t_{f,h})(\gamma_{UV}\otimes1_W)&&,\label{eq:t_f(gh)}
\end{alignat}
\end{subequations}
and
\begin{subequations}\label{subeq:right inf hex}
\begin{alignat}{6}
t_{(UV)W}=&~1_U\otimes t_{VW}+(1_U\otimes\gamma_{WV})(t_{UW}\otimes1_V)(1_U\otimes\gamma_{VW})\quad&&,\label{eq:t_(UV)W}\\
t_{f\otimes g,h}=&~f\otimes t_{g,h}+(1_{U'}\otimes\gamma_{W'V'})(t_{f,h}\otimes g)(1_U\otimes\gamma_{VW})&&.\label{eq:t_(fg)h}
\end{alignat}
\end{subequations}
An infinitesimal 2-braiding is \textbf{symmetric} (or, $\gamma$-\textbf{equivariant}) if it intertwines with the symmetric braiding $\gamma$, i.e.:
\begin{equation}
\gamma_{U,V}\,t_{U,V}=t_{V,U}\,\gamma_{U,V}\qquad,\qquad\gamma_{U',V'}\,t_{f,g}=t_{g,f}\,\gamma_{U,V}\quad.
\end{equation}
\end{defi}
We denote \eqref{subeq:left inf hex} and \eqref{subeq:right inf hex} as, respectively,
\begin{equation}\label{subeq:index inf hex}
t_{1(23)}=t_{12}+t_{13}\qquad,\qquad t_{(12)3}=t_{13}+t_{23}\quad.
\end{equation}
In the ordinary context of 1-category theory, \textit{naturality} of an infinitesimal braiding $t$ implies that it satisfies the \textbf{four-term relations},
\begin{equation}\label{eq:four-term relations}
[t_{12},t_{13}+t_{23}]=0=[t_{23},t_{12}+t_{13}]\quad.
\end{equation}
In our context, \textit{pseudonaturality} of an infinitesimal 2-braiding $t$ obstructs the four-term relations in a very specific way. 
\begin{defi}\label{defi: modification}
Given pseudonatural transformations $\xi,\xi':F\Rightarrow G:\C\rightarrow\D$, a \textbf{modification} $\Xi:\xi\Rrightarrow\xi'$ consists of, for each object $U\in\C$, a morphism $\Xi_U\in\D[F(U),G(U)]^{-1}$ such that
\begin{subequations}
\begin{equation}\label{eq:mod components are homotopies}
\partial(\Xi_U)=\xi_U-\xi'_U
\end{equation}
and, for every $f\in\C[U,V]$,
\begin{equation}\label{eqn:mod single condition}
\Xi_VF(f)+\xi_f=\xi'_f+G(f)\,\Xi_U\quad.
\end{equation}
\end{subequations}
\end{defi}
The obstruction to the four-term relations is a special modification, one witnessing the lack of exchange between the two different compositions of pseudonatural transformations. To be specific, the \textbf{vertical composition}
\begin{subequations}\label{eqn: ver comp of pseudos}
\begin{equation}
\begin{tikzcd}
	{\C} && {\D}
	\arrow[""{name=0, anchor=center, inner sep=0}, "G"{description}, from=1-1, to=1-3]
	\arrow[""{name=1, anchor=center, inner sep=0}, "F", curve={height=-32pt}, from=1-1, to=1-3]
	\arrow[""{name=2, anchor=center, inner sep=0}, "H"',curve={height=32pt}, from=1-1, to=1-3]
	\arrow["\xi\,"', shorten <=5pt, shorten >=5pt, Rightarrow, from=1, to=0]
	\arrow["\theta\,"', shorten <=5pt, shorten >=5pt, Rightarrow, from=0, to=2]
\end{tikzcd}~~\stackrel{\circ}{\longmapsto}~~
\begin{tikzcd}
	{\C} && {\D}
	\arrow[""{name=1, anchor=center, inner sep=0}, "F", curve={height=-32pt}, from=1-1, to=1-3]
	\arrow[""{name=2, anchor=center, inner sep=0}, "H"', curve={height=32pt}, from=1-1, to=1-3]
	\arrow["\theta\xi\,"', shorten <=5pt, shorten >=5pt, Rightarrow, from=1, to=2]
\end{tikzcd}
\end{equation}
is defined by setting, for $f\in\C[U,V]$ :
\begin{equation}
(\theta\xi)_U:=\theta_U\xi_U\qquad,\qquad(\theta\xi)_f:=\theta_f\xi_U+\theta_V\xi_f
\end{equation}
\end{subequations}
whereas the \textbf{horizontal composition}
\begin{subequations}\label{subeq:horcomp of pseudos}
\begin{equation}
\begin{tikzcd}
	{\C} && {\D} && {\E}
	\arrow[""{name=0, anchor=center, inner sep=0}, "F", curve={height=-32pt}, from=1-1, to=1-3]
	\arrow[""{name=1, anchor=center, inner sep=0}, "G"', curve={height=32pt}, from=1-1, to=1-3]
	\arrow[""{name=2, anchor=center, inner sep=0}, "F'", curve={height=-32pt}, from=1-3, to=1-5]
	\arrow[""{name=3, anchor=center, inner sep=0}, "G'"', curve={height=32pt}, from=1-3, to=1-5]
	\arrow["\xi\,"', shorten <=6pt, shorten >=6pt, Rightarrow, from=0, to=1]
	\arrow["\upsilon\,"', shorten <=6pt, shorten >=6pt, Rightarrow, from=2, to=3]
\end{tikzcd}
~~\stackrel{*}{\longmapsto}~~
\begin{tikzcd}
	{\C} && {\E}
	\arrow[""{name=0, anchor=center, inner sep=0}, "F'F", curve={height=-32pt}, from=1-1, to=1-3]
	\arrow[""{name=1, anchor=center, inner sep=0}, "G'G'"', curve={height=32pt}, from=1-1, to=1-3]
	\arrow["\upsilon\,*\,\xi\,"', shorten <=6pt, shorten >=6pt, Rightarrow, from=0, to=1]
\end{tikzcd}
\end{equation}
is defined by setting:
\begin{equation}
(\upsilon*\xi)_U:=\upsilon_{G(U)}F'(\xi_U)\qquad,\qquad
(\upsilon*\xi)_f:=\upsilon_{G(f)}F'(\xi_U)
+\upsilon_{G(V)}F'(\xi_f)\quad.
\end{equation}
\end{subequations}
The vertical composition \eqref{eqn: ver comp of pseudos} and horizontal composition \eqref{subeq:horcomp of pseudos} are associative and admit the obvious units $\Id_F$ and $\Id_{\id_\C}$, respectively. A \textbf{pseudonatural isomorphism} $\xi:F\Rightarrow G:\C\to\D$ is one which admits an inverse $\xi^{-1}:G\Rightarrow F:\C\to\D$ under vertical composition \eqref{eqn: ver comp of pseudos}.
\begin{defi}\label{def:exchanger}
Given any composable diagram of the form
\begin{subequations}\label{eqn:compositioncoherences}
\begin{equation}
% https://q.uiver.app
\begin{tikzcd}
	{\C} && {\D} && {\E}
	\arrow[""{name=0, anchor=center, inner sep=0}, "F", curve={height=-32pt}, from=1-1, to=1-3]
	\arrow[""{name=1, anchor=center, inner sep=0}, "G"{description}, from=1-1, to=1-3]
	\arrow[""{name=2, anchor=center, inner sep=0}, "H"', curve={height=32pt}, from=1-1, to=1-3]
	\arrow[""{name=3, anchor=center, inner sep=0}, "{F'}", curve={height=-32pt}, from=1-3, to=1-5]
	\arrow[""{name=4, anchor=center, inner sep=0}, "{H'}"', curve={height=32pt}, from=1-3, to=1-5]
	\arrow[""{name=5, anchor=center, inner sep=0}, "{G'}"{description}, from=1-3, to=1-5]
	\arrow["\xi\,"', shorten <=5pt, shorten >=5pt, Rightarrow, from=0, to=1]
	\arrow["\theta\,"', shorten <=5pt, shorten >=5pt, Rightarrow, from=1, to=2]
	\arrow["{\upsilon\,}"', shorten <=5pt, shorten >=5pt, Rightarrow, from=3, to=5]
	\arrow["{\lambda\,}"', shorten <=5pt, shorten >=5pt, Rightarrow, from=5, to=4]
\end{tikzcd}\quad,
\end{equation}
the \textbf{exchanger} is the modification
\begin{equation}
*^2_{\lambda,\theta|\upsilon,\xi}\,:\,(\lambda*\theta)(\upsilon*\xi)~\Rrightarrow~\lambda\upsilon*\theta\xi\qquad,\qquad\left(*^2_{\lambda,\theta|\upsilon,\xi}\right)_U\,:=\,\lambda_{H(U)}\,\upsilon_{\theta_U}F'(\xi_U)\quad.
\end{equation}
\end{subequations}
\end{defi}
The modifications of Definition \ref{defi: modification} admit three different levels of composition though we only describe in detail the highest two: 
\begin{enumerate}
\item[(i)] Given $\xi\xRrightarrow{\Xi}\xi'\xRrightarrow{\Xi'}\xi'':F\Rightarrow G:\C\rightarrow\D$, the \textbf{lateral composition} $\Xi'\cdot\Xi:\xi\Rrightarrow\xi''$ is defined by setting $(\Xi'\cdot\Xi)_U:=\Xi'_U+\Xi_U$ for all $U\in\C$. This composition is associative, unital with respect to the \textbf{vanishing modifications} $0:\xi\Rrightarrow\xi$, and invertible with respect to the \textbf{reverse} $\overleftarrow{\Xi}:\xi'\Rrightarrow\xi$ defined by $\overleftarrow{\Xi}_U:=-\Xi_U$.
\item[(ii)] Given $\Xi:\xi\Rrightarrow\xi':F\Rightarrow G:\C\rightarrow\D$ and $\Theta:\theta\Rrightarrow\theta':G\Rightarrow H:\C\rightarrow\D$, the \textbf{vertical composition} $\Theta\Xi:\theta\xi\Rrightarrow\theta'\xi':F\Rightarrow H$ is defined by setting $(\Theta\Xi)_U:=\Theta_U\xi'_U+\theta_U\Xi_U$. This composition is also associative and unital. We define the \textbf{whiskering} of $\Theta$ by $\xi$ as the modification $\Theta\xi:\theta\xi\Rrightarrow\theta'\xi$ with components $(\Theta\xi)_U:=\Theta_U\xi_U$ and likewise for the whiskering of $\Xi$ by $\theta$. A modification $\Xi$ is invertible under the vertical composition if and only if both $\xi$ and $\xi'$ are pseudonatural isomorphisms in which case the inverse is given by $\Xi^{-1}:=\xi^{-1}\overleftarrow{\Xi}\xi'^{-1}:\xi^{-1}\Rrightarrow\xi'^{-1}:G\Rightarrow F$. 
\end{enumerate}
\begin{rem}\label{rem:comp of mods}
The horizontal composition of modifications is also associative and unital. Furthermore, the vertical composition $\circ$ is functorial and the horizontal composition $*$ is a strictly-unitary pseudofunctor hence we have a tricategory $\dgCat^{[-1,0],\ps}$ of $\Ch^{[-1,0]}$-categories, $\Ch^{[-1,0]}$-functors, pseudonatural transformations and modifications. The \textit{only} weakness of this tricategory is given by the nontriviality of the exchanger from Definition \ref{def:exchanger}.   
\end{rem}
In order to determine the exchanger which obstructs the four-term relations, we still have to describe the \textbf{monoidal composition}
\begin{subequations}\label{subeq:moncomp of pseudos}
\begin{equation}
\begin{tikzcd}
	{\C} && {\C'}
	\arrow[""{name=0, anchor=center, inner sep=0}, "F", curve={height=-32pt}, from=1-1, to=1-3]
	\arrow[""{name=1, anchor=center, inner sep=0}, "{F'}"', curve={height=32pt}, from=1-1, to=1-3]
	\arrow["\xi\,"', shorten <=6pt, shorten >=6pt, Rightarrow, from=0, to=1]
\end{tikzcd}
~~
\begin{tikzcd}
	{\D} && {\D'}
	\arrow[""{name=0, anchor=center, inner sep=0}, "G", curve={height=-32pt}, from=1-1, to=1-3]
	\arrow[""{name=1, anchor=center, inner sep=0}, "{G'}"', curve={height=32pt}, from=1-1, to=1-3]
	\arrow["\upsilon\,"', shorten <=6pt, shorten >=6pt, Rightarrow, from=0, to=1]
\end{tikzcd}
~~\stackrel{\smallbox}{\longmapsto}~~
\begin{tikzcd}
	{\C\smallbox\D} && {\C'\smallbox\D'}
	\arrow[""{name=0, anchor=center, inner sep=0}, "F\smallbox G", curve={height=-32pt}, from=1-1, to=1-3]
	\arrow[""{name=1, anchor=center, inner sep=0}, "{F'\smallbox G'}"', curve={height=32pt}, from=1-1, to=1-3]
	\arrow["\xi\smallbox\upsilon\,"', shorten <=6pt, shorten >=6pt, Rightarrow, from=0, to=1]
\end{tikzcd}
\end{equation}
which is defined by setting, for $f\in\C[U,U']$ and $g\in\D[V,V']$ :
\begin{equation}
(\xi\smallbox\upsilon)_{UV}:=\xi_U\boxtimes\upsilon_V\qquad,\qquad(\xi\smallbox\upsilon)_{f,g}:=
\xi_f\boxtimes G'(g)\upsilon_V + \xi_{U'}F(f)\boxtimes \upsilon_g\quad.
\end{equation}
\end{subequations}
Modifications also admit a monoidal composition; altogether, $\smallbox$ is an associative unital 3-functor hence $\dgCat^{[-1,0],\ps}$ is a monoidal tricategory. Furthermore, the symmetric braiding $\tau$ on $\Ch^{[-1,0]}$ provides a symmetric braiding on $\dgCat^{[-1,0],\ps}$. Lastly, we mention that $\dgCat^{[-1,0],\ps}$ is actually a \textit{closed} symmetric monoidal tricategory given that pseudonatural transformations and modifications can be added and scaled, while a modification $\Xi:\xi\Rrightarrow\xi'$ can be differentiated to a pseudonatural transformation $\partial(\Xi):=\xi-\xi'$.
\begin{rem}
By abuse of notation, we will often denote the lateral composition of modifications as $\Xi'\cdot\Xi=\Xi'+\Xi=\Xi+\Xi'$ and the reverse as $\overleftarrow{\Xi}=-\Xi$ even though those modifications have different (co)domains; the context will make it clear which is being used.
\end{rem}
We can now use the linearity of pseudonatural transformations together with the above three compositions to rewrite \eqref{subeq:left inf hex} as
\begin{subequations}\label{subeq:pre-index inf hex}
\begin{equation}\label{eq:pre-index t_U(VW)}
t*\Id_{\id_\C\smallbox\otimes}=\Id_\otimes*\big(t\smallbox\Id_{\id_\C}\big)+\Id_\otimes*\Big(\big[\gamma^{-1}\smallbox\Id_{\id_\C}\big]\big[(\Id_{\id_\C}\smallbox t)*\Id_{\tau_{\C,\C}\smallbox\id_\C}\big]\big[\gamma\smallbox\Id_{\id_\C}\big]\Big)~
\end{equation}
and \eqref{subeq:right inf hex} as
\begin{equation}\label{eq:pre-index t_(UV)W}
t*\Id_{\otimes\smallbox\id_\C}=\Id_\otimes*\big(\Id_{\id_\C}\smallbox t\big)+\Id_\otimes*\Big(\big[\Id_{\id_\C}\smallbox\gamma^{-1}\big]\big[(t\smallbox\Id_{\id_\C})*\Id_{\id_\C\smallbox\tau_{\C,\C}}\big]\big[\Id_{\id_\C}\smallbox\gamma\big]\Big)~.
\end{equation}
\end{subequations}
In fact, \eqref{subeq:pre-index inf hex} is the precise meaning behind the index notation \eqref{subeq:index inf hex}. Where possible, we make use of this much simpler index notation, e.g. consider the composable diagram 
\begin{subequations}
\begin{equation}
\begin{tikzcd}
	& {} && {} \\
	{\C\smallbox\C\smallbox\C} && {\C\smallbox\C} && {\C} \\
	& {} && {}
	\arrow["{\otimes\smallbox\id_\C}", curve={height=-30pt}, from=2-1, to=2-3]
	\arrow["{\otimes\smallbox\id_\C}"', curve={height=30pt}, from=2-1, to=2-3]
	\arrow[""{name=0, anchor=center, inner sep=0}, "{\otimes\smallbox\id_\C}"{description}, from=2-1, to=2-3]
	\arrow["{\otimes}", curve={height=-30pt}, from=2-3, to=2-5]
	\arrow["{\otimes}"', curve={height=30pt}, from=2-3, to=2-5]
	\arrow[""{name=1, anchor=center, inner sep=0}, "{\otimes}"{description}, from=2-3, to=2-5]
	\arrow["{\Id_{\otimes\smallbox\id_\C}\,}"', shorten <=7pt, shorten >=7pt, Rightarrow, from=1-2, to=0]
	\arrow["t\,"', shorten <=7pt, shorten >=7pt, Rightarrow, from=1-4, to=1]
	\arrow["{t\smallbox\Id_{\id_\C}\,}"', shorten <=7pt, shorten >=7pt, Rightarrow, from=0, to=3-2]
	\arrow["\Id_\otimes\,"', shorten <=7pt, shorten >=7pt, Rightarrow, from=1, to=3-4]
\end{tikzcd}\quad,
\end{equation}
Definition \ref{def:exchanger} tells us that the exchanger \eqref{eqn:compositioncoherences} takes on the specific form
\begin{equation}\label{eq:ast left 4T relationator}
*^2_{\Id_\otimes,\,t\smallbox\Id_{\id_\C}|\,t\,,\,\Id_{\otimes\smallbox\id_\C}}:t_{12}t_{(12)3}\Rrightarrow t_{(12)3}t_{12}\quad.
\end{equation}
Using the linearity of modifications, we can rewrite \eqref{eq:ast left 4T relationator} as the \textbf{left four-term relationator}
\begin{equation}
\L:[t_{12},t_{13}+t_{23}]\Rrightarrow0
\end{equation}
which has components
\begin{equation}\label{eq:L_123:=}
\L_{UVW}=t_{t_{UV},1_W}\quad.
\end{equation}
\end{subequations}
Similarly, we also have a \textbf{right four-term relationator}
\begin{subequations}
\begin{equation}
\R:[t_{23},t_{12}+t_{13}]\Rrightarrow0
\end{equation}
which has components 
\begin{equation}\label{eq:R_123:=}
\R_{UVW}=t_{1_U,t_{VW}}\quad.
\end{equation}
\end{subequations}
Using \eqref{eq:mod components are homotopies}, we see the specific way these modifications obstruct the four-term relations:
\begin{subequations}
\begin{align}
\partial(\L_{UVW})=&~(t_{UV}\otimes1_W)t_{(UV)W}-t_{(UV)W}(t_{UV}\otimes1_W)\quad,\label{eq:deformed left 4T relation}\\
\partial(\R_{UVW})=&~(1_U\otimes t_{VW})t_{U(VW)}-t_{U(VW)}(1_U\otimes t_{VW})\quad.\label{eq:deformed right 4T relation}
\end{align}
\end{subequations}
Given a symmetric infinitesimal 2-braiding $t$ on a symmetric strict monoidal $\Ch_\bbC^{[-1,0]}$-category $(\C,\otimes,I,\gamma)$ and a deformation parameter $\hbar$, one chooses an ansatz braiding as $\sigma=\gamma\,e^{i\pi\hbar t}$ and an ansatz associator as $\alpha=\Phi(t_{12},t_{23})$, where $\Phi$ is Drinfeld's KZ series (see \cite[Proposition XIX.6.4]{Kassel}, \cite[Theorem 20]{BRW} or Definition \ref{def:Drinfeld KZ series} below). As shown in \cite[Section 5]{Kem25a}, the four-term relationators $\L$ and $\R$ complicate the usual deformation quantisation story by obstructing the hexagon axiom already at second order in $\hbar$. Anticipating these obstructions forced us to go one level higher in category theory thus we introduced the definition \cite[Definition 2.25]{Kem25a} of a braided (strictly-unital) monoidal $\Ch^{[-1,0]}$-category $(\C,\otimes,I,\alpha,\sigma,\Pi,\H^L,\H^R)$ by specifying the definition of a braided monoidal bicategory (as in \cite[Definition C.2]{Schommer} or \cite[Definition 12.1.6]{Yau}) to our context provided by $\dgCat^{[-1,0],\ps}$. In particular, the associator $\alpha$ and braiding $\sigma$ do not satisfy the usual pentagon and hexagon axioms, instead these are obstructed by the \textbf{pentagonator} $\Pi$ and \textbf{hexagonator} $\H^{L/R}$ modifications, i.e.:
\begin{subequations}\label{subeq:pentagonator and hexagonators}
\begin{alignat}{6}
\Pi~:&~\alpha_{234}\,\alpha_{1(23)4}\,\alpha_{123}&&\xRrightarrow{~~~}\alpha_{12(34)}\,\alpha_{(12)34}\quad&&&,\label{eq:index pentagonator}\\
\H^L:&\quad\alpha_{231}\,\sigma_{1(23)}\,\alpha&&\xRrightarrow{~~~}~~\sigma_{13}\,\alpha_{213}\,\sigma_{12}&&&,\label{eq:index left hexagonator}\\
\H^R:&\quad\alpha^{-1}_{312}\,\sigma_{(12)3}\,\alpha^{-1}&&\xRrightarrow{~~~}~~\sigma_{13}\,\alpha^{-1}_{132}\,\sigma_{23}&&&.\label{eq:index right hexagonator}
\end{alignat}
\end{subequations}
The data \eqref{subeq:pentagonator and hexagonators} is subject to five higher coherence conditions, all of which state that their associated expression must vanish: $$~$$
(i) \textbf{Associahedron},
\begin{subequations}\label{subeq:axioms of bra mon 2cat}
\begin{align*}
&\Big(\alpha_{345}\Pi_{12(34)5}+\alpha_{\alpha_{345}}\alpha_{(12)(34)5}\Big)\alpha_{(12)34}+\Big(\alpha_{345}\alpha_{2(34)5}\alpha_{\alpha_{234}}-\Pi_{2345}\alpha_{1((23)4)5}\Big)\alpha_{1(23)4}\alpha_{123}\\&+\alpha_{345}\alpha_{2(34)5}\alpha_{1(2(34))5}\Pi_{1234}+\alpha_{23(45)}\Big(\alpha_{1(23)(45)}\alpha_{\alpha_{123}}-\Pi_{1(23)45}\alpha_{123}\Big)-\Pi_{123(45)}\alpha_{((12)3)45}\\&+\alpha_{12(3(45))}\Pi_{(12)345}\quad.\numberthis\label{eq:associahedron}
\end{align*}
(ii) \textbf{Left tetrahedron}, 
\begin{align*}
&\Big(\left[\alpha_{341}\alpha_{2(34)1}\sigma_{\alpha_{234}}-\Pi_{2341}\sigma_{1((23)4)}\right]\alpha_{1(23)4}+\left[\sigma_{14}\Pi_{2314}+\alpha_{\sigma_{14}}\alpha_{(23)14}\right]\sigma_{1(23)}\Big)\alpha_{123}\\&-\alpha_{23(41)}\H^L_{1(23)4}\alpha_{123}+\left(\sigma_{14}\alpha_{314}\alpha_{\sigma_{13}}+\H^L_{134}\alpha_{2(13)4}\right)\alpha_{213}\sigma_{12}-\sigma_{14}\alpha_{314}\alpha_{2(31)4}\H^L_{123}\\&+\alpha_{341}\left(\H^L_{12(34)}\alpha_{(12)34}+\sigma_{1(34)}\left[\alpha_{21(34)}\alpha_{\sigma_{12}}-\Pi_{2134}\sigma_{12}\right]+\alpha_{2(34)1}\sigma_{1(2(34))}\Pi\right)\quad.\numberthis\label{eq:left tetrahedron}
\end{align*}
(iii) \textbf{Right tetrahedron},
\begin{align*}
&\Big(\Big[\alpha^{-1}_{412}\alpha^{-1}_{4(12)3}\sigma_{\alpha^{-1}_{123}}-\Pi^{-1}_{4123}\sigma_{(1(23))4}\Big]\alpha^{-1}_{1(23)4}+\Big[\sigma_{14}\Pi^{-1}_{1423}+\alpha^{-1}_{\sigma_{14}}\alpha^{-1}_{14(23)}\Big]\sigma_{(23)4}\Big)\alpha^{-1}_{234}\\&+\alpha^{-1}_{(41)23}\H^R_{1(23)4}\alpha^{-1}_{234}+\sigma_{14}\alpha^{-1}_{142}\Big(\alpha^{-1}_{\sigma_{24}}\alpha^{-1}_{243}\sigma_{34}+\alpha^{-1}_{1(42)3}\H^R_{234}\Big)+\H^R_{124}\alpha^{-1}_{1(24)3}\alpha^{-1}_{243}\sigma_{34}\\&+\alpha^{-1}_{412}\Big(\H^R_{(12)34}\alpha^{-1}_{12(34)}+\sigma_{(12)4}\Big[\alpha^{-1}_{(12)43}\alpha^{-1}_{\sigma_{34}}-\Pi^{-1}_{1243}\sigma_{34}\Big]+\alpha^{-1}_{4(12)3}\sigma_{((12)3)4}\Pi^{-1}_{1234}\Big)\quad.\numberthis\label{eq:right tetrahedron}
\end{align*}
(iv) \textbf{Hexahedron},
\begin{align*}
&\sigma_{14}\alpha^{-1}_{142}\sigma_{24}\alpha_{31(24)}\Big(\alpha_{\sigma_{13}}\alpha^{-1}_{132}\sigma_{23}-\alpha_{(31)24}\H^R_{123}\Big)-\alpha^{-1}_{412}\Big(\Pi_{3412}\alpha^{-1}_{(34)12}\sigma_{(12)(34)}\alpha_{(12)34}+\H^L_{(12)34}\Big)\alpha^{-1}_{123}\\&+\alpha_{3(41)2}\alpha_{341}\Big(\sigma_{1(34)}\alpha_{134}\Big[\alpha^{-1}_{(13)42}\alpha^{-1}_{13(42)}\H^L_{234}+\Pi^{-1}_{1342}\alpha_{342}\sigma_{2(34)}\alpha_{234}\Big]\alpha_{1(23)4}+\H^R_{12(34)}\alpha_{234}\alpha_{1(23)4}\\&-\alpha^{-1}_{(34)12}\sigma_{(12)(34)}\alpha^{-1}_{12(34)}\Pi_{1234}\alpha^{-1}_{123}\Big)-\Big(\alpha^{-1}_{412}\sigma_{(12)4}\alpha^{-1}_{124}\Pi_{3124}+\H^R_{124}\alpha_{31(24)}\alpha_{(31)24}\Big)\alpha^{-1}_{312}\sigma_{(12)3}\alpha^{-1}_{123}\\&+\Big(\Big[\sigma_{14}\alpha^{-1}_{142}\Pi_{3142}-\alpha_{\sigma_{14}}\alpha_{314}\Big]\alpha^{-1}_{(31)42}\sigma_{24}-\sigma_{14}\alpha^{-1}_{142}\alpha_{\sigma_{24}}\Big)\sigma_{13}\alpha_{(13)24}\alpha^{-1}_{132}\sigma_{23}\\&+\alpha_{3(41)2}\Big(\sigma_{14}\alpha_{314}\Big[\alpha^{-1}_{\sigma_{13}}\sigma_{24}\alpha^{-1}_{13(24)}-\sigma_{13}\alpha^{-1}_{(13)42}\alpha^{-1}_{\sigma_{24}}\Big]+\H^L_{134}\alpha^{-1}_{(13)42}\alpha^{-1}_{13(42)}\sigma_{24}\Big)\alpha_{324}\sigma_{23}\alpha_{1(23)4}\\&+\alpha_{3(41)2}\sigma_{14}\alpha_{314}\alpha^{-1}_{(31)42}\sigma_{13}\sigma_{24}\alpha^{-1}_{13(24)}\Big(\Pi_{1324}\alpha^{-1}_{132}\sigma_{23}+\alpha_{324}\alpha_{\sigma_{23}}\Big)\qquad.\numberthis\label{eq:Hexahedron}
\end{align*}
(v) \textbf{Breen polytope},
\begin{equation}\label{eq:Breen axiom index version}
\sigma_{\sigma_{12}}+\alpha_{321}\Big(\H^R_{213}\alpha_{213}\sigma_{12}-\sigma_{23}\alpha_{231}^{-1}\H^L+\sigma_{\sigma_{23}}\alpha\Big)+\Big(\H^L_{132}\alpha_{132}^{-1}\sigma_{23}-\sigma_{12}\alpha_{312}\H^R\Big)\alpha\quad.
\end{equation}
\end{subequations}
As shown in \cite[Proposition 5.14]{Kem25a}, the ansatz associator
\begin{equation}\label{eq:alpha at 2nd order}
\alpha=\Phi(t_{12},t_{23})=1-\tfrac{1}{6}\pi^2\hbar^2[t_{12},t_{23}]+\O\big(\hbar^3\big)
\end{equation}
satisfies the pentagon axiom up to and including order $\hbar^2$ thus we can choose a vanishing pentagonator, doing so satisfies the associahedron axiom \eqref{eq:associahedron} up to and including order $\hbar^2$. Substituting the ansatz associator \eqref{eq:alpha at 2nd order} together with the ansatz braiding
\begin{equation}
\sigma=\gamma e^{i\pi\hbar t}=\gamma\Big(1+i\pi\hbar t-\tfrac{1}{2}\pi^2\hbar^2t^2+\O\big(\hbar^3\big)\Big)
\end{equation}
into \eqref{subeq:pentagonator and hexagonators} gives, for a \textit{symmetric} infinitesimal 2-braiding:
\begin{subequations}\label{subeq:hexagonators at 2nd order}
\begin{alignat}{6}
-\tfrac{1}{6}\pi^2\hbar^2\gamma_{1(23)}(2\L+\R)+\O\big(\hbar^3\big):~&\alpha_{231}\,\sigma_{1(23)}\,\alpha&&\xRrightarrow{~~~}\sigma_{13}\,\alpha_{213}\,\sigma_{12}\quad&&&,\\
-\tfrac{1}{6}\pi^2\hbar^2\gamma_{(12)3}(\L+2\R)+\O\big(\hbar^3\big):~&\alpha^{-1}_{312}\,\sigma_{(12)3}\,\alpha^{-1}&&\xRrightarrow{~~~}\sigma_{13}\,\alpha^{-1}_{132}\,\sigma_{23}&&&.
\end{alignat}
\end{subequations}
As shown in \cite[Section 5.2]{Kem25a}, the modifications \eqref{subeq:hexagonators at 2nd order} will not necessarily satisfy the four axioms \eqref{eq:left tetrahedron}-\eqref{eq:Breen axiom index version} but they will if the symmetric infinitesimal 2-braiding $t$ satisfies some extra conditions discovered by Cirio and Martins \cite{Joao}.
\begin{defi}\label{def:coherent t}
Given a symmetric strict monoidal $\Ch^{[-1,0]}$-category $(\C,\otimes,I,\gamma)$, a \textbf{coherent} infinitesimal 2-braiding $t$ satisfies, for all $U,V,W\in\C$,
\begin{equation}
-(\gamma_{VU}\otimes1_W)\R_{VUW}(\gamma_{UV}\otimes1_W)=\L_{UVW}+\R_{UVW}=-(1_U\otimes\gamma_{WV})\L_{UWV}(1_U\otimes\gamma_{VW})\quad.
\end{equation}
A symmetric infinitesimal 2-braiding $t$ is \textbf{totally symmetric} if, for all $U,V,W\in\C$,
\begin{equation}
t_{\gamma_{UV},1_W}=0\quad.
\end{equation}
\end{defi}
Knowing that these properties are sufficient to solve the deformation quantisation problem at order $\hbar^2$, we assume a coherent totally symmetric infinitesimal 2-braiding $t$ and look for modifications of the form:
\begin{subequations}
\begin{equation}\label{eq:symstr pentagonator}
\Pi:\Phi(t_{23},t_{34})\Phi(t_{12}+t_{13},t_{24}+t_{34})\Phi(t_{12},t_{23})\xRrightarrow{~~~}\Phi(t_{12},t_{23}+t_{24})\Phi(t_{13}+t_{23},t_{34})
\end{equation}
and
\begin{equation}\label{eq:symstr right pre-hex}
R:\Phi(t_{12}\,,t_{13})e^{i\pi\hbar t_{(12)3}}\Phi(t_{23}\,,t_{12})\xRrightarrow{~~~}e^{i\pi\hbar t_{13}}\Phi(t_{23}\,,t_{13})e^{i\pi\hbar t_{23}}\quad.
\end{equation}
As in \cite[Remark 5.22]{Kem25a}, a totally symmetric infinitesimal 2-braiding gives us 
\begin{equation}\label{eq:symstr left pre-hex}
R_{321}:\Phi(t_{23},t_{13})e^{i\pi\hbar t_{1(23)}}\Phi(t_{12},t_{23})\xRrightarrow{~~~}e^{i\pi\hbar t_{13}}\Phi(t_{12},t_{13})e^{i\pi\hbar t_{12}}
\end{equation}
\end{subequations}
thus we define candidate hexagonators as $\H^R:=\gamma_{(12)3}R\,$ and $\H^L:=\gamma_{1(23)}L$, where $L:=R_{321}$. 

\section{The fundamental conjecture}\label{subsec:conjecture}
This section introduces the notion of Drinfeld-Kohno 2-algebras in Definition \ref{def:nth DK 2-algebra} so that we may state our fundamental conjecture. As explained in Remark \ref{rem:conjecture is fundamental}, should Conjecture \ref{conj:fundamental} be true then the latter two sections offer a self-contained solution to the problem of integrating infinitesimal 2-braidings to provide a concrete braided monoidal 2-category. In other words, the candidate hexagonator series of Theorem \ref{theo:hexagonator} and the candidate pentagonator series of Theorem \ref{theo:pentagonator} automatically satisfy their axioms \eqref{subeq:axioms of bra mon 2cat} if Conjecture \ref{conj:fundamental} is true.
\sk

See \cite[Definition 2.5]{Chen} for the following definition.
\begin{defi}\label{def:associative 2-algebra}
An \textbf{associative 2-algebra} consists of three pieces of data:
\begin{enumerate}
\item[(i)] A pair of associative algebras $A$ and $B$.
\item[(ii)] An algebra homomorphism $\partial:B\to A$.
\item[(iii)] An $A$-bimodule structure on $B$, i.e. for all $a,a'\in A$ and $b\in B$,
\begin{equation}
(a'a)b=a'(ab)\qquad,\qquad(a'b)a=a'(ba)\qquad,\qquad(ba')a=b(a'a)\quad.
\end{equation}
\end{enumerate}
These three pieces of data are required to satisfy the following two axioms:
\begin{enumerate}
\begin{subequations}\label{subeq:axioms of assoc 2-algebra}
\item \textbf{Two-sided $A$-equivariance} of $\partial$, i.e. for all $a\in A$ and $b\in B$,
\begin{equation}
\partial(ab)=a\partial(b)\qquad,\qquad\partial(ba)=\partial(b)a\quad.
\end{equation}
\item The \textbf{Peiffer identity}, i.e. for all $b,b'\in B$,
\begin{equation}
\partial(b')b=b'b=b'\partial(b)\quad.
\end{equation}
\end{subequations}
\end{enumerate}
\end{defi}
Analogous to \cite[Construction 3.22]{Kem25b}, given a pair of $\Ch^{[-1,0]}$-categories $\BB$ and $\C$ together with a $\Ch^{[-1,0]}$-functor $F:\BB\to\C$, we define an associative 2-algebra $\End_F$ as follows:
\begin{enumerate}
\item[(i)] Consider the associative algebras of pseudonatural transformations of the form $\xi:F\Rightarrow F$ and modifications of the form $\Xi:\xi\Rrightarrow0:F\Rightarrow F$ with multiplication given by the vertical composition.
\item[(ii)] Given $\Xi:\xi\Rrightarrow0:F\Rightarrow F$, setting $\partial(\Xi):=\xi$ obviously defines an algebra homomorphism.
\item[(iii)] The modifications form a bimodule over the pseudonatural transformations via whiskering. 
\end{enumerate}
The above data evidently satisfies the axioms \eqref{subeq:axioms of assoc 2-algebra}. 
\begin{ex}\label{ex:End_otimes^n}
Given a natural number $n\in\bbN$ and a symmetric strict monoidal $\Ch^{[-1,0]}$-category $(\C,\otimes,I,\gamma)$, we set $\BB=\C^{\boxtimes(n+1)}$ and $F=\otimes^n:\C^{\boxtimes(n+1)}\to\C$ thus the associative 2-algebra $\End_{\otimes^n}$ consists of pseudonatural transformations of the form $\xi:\otimes^n\Rightarrow\otimes^n$ and modifications of the form $\Xi:\xi\Rrightarrow0:\otimes^n\Rightarrow\otimes^n$.
\end{ex}
Let us consider the special case $n=2$ of Example \ref{ex:End_otimes^n} and suppose we are given an infinitesimal 2-braiding $t$. Consider the modification
\begin{equation}\label{eq:nonvanishing modification}
\L_{213}-\L:[t_{21},t_{23}+t_{13}]-[t_{12},t_{13}+t_{23}]\xRrightarrow{~~~}0\quad,
\end{equation}
if our infinitesimal 2-braiding $t$ is symmetric then $t_{21}=t_{12}$ and the domain of \eqref{eq:nonvanishing modification} is 0 yet, for $U,V,W\in\C$,
\begin{equation}
(\gamma_{VU}\otimes1_W)t_{t_{VU},1_W}(\gamma_{UV}\otimes1_W)\neq t_{t_{UV},1_W}\quad,
\end{equation}
in general. Conversely, if $t$ is totally symmetric then \cite[Lemma 5.21]{Kem25a} gives us:
\begin{align}
\L=\R_{312}=\L_{213}=\R_{321}\quad,\quad\R=\L_{231}=\R_{132}=\L_{321}\quad,\quad\L_{132}=\R_{213}=\L_{312}=\R_{231}
\end{align}
while the modification 
\begin{equation}
\L+\R+\L_{132}:[t_{12},t_{13}+t_{23}]+[t_{23},t_{12}+t_{13}]+[t_{13},t_{12}+t_{32}]\xRrightarrow{~~~}0
\end{equation}
also has domain 0 yet does not vanish unless $t$ is, further, coherent. Let us now turn our attention to the special case $n=3$ of Example \ref{ex:End_otimes^n}. 
\begin{lem}\label{lem:5 relations}
Given an infinitesimal 2-braiding $t$ on a symmetric strict monoidal $\Ch^{[-1,0]}$-category $(\C,\otimes,I,\gamma)$, we have the following five relations:
\begin{subequations}\label{subeq:5 relations}
\begin{alignat}{6}
[t_{(123)4},\R_{123}]-[t_{1(23)},\L_{234}]+[t_{23},\L_{1(23)4}]=&~0&&,\label{eq:1/5 relation}\\
[t_{1(234)},\R_{234}]+[t_{34},\R_{12(34)}]-[t_{2(34)},\R_{134}]=&~0&&,\label{eq:2/5 relation}\\
[t_{(123)4},\L_{123}]+[t_{12},\L_{(12)34}]-[t_{(12)3},\L_{124}]=&~0\quad&&,\label{eq:3/5 relation}\\
[t_{1(234)},\L_{234}]+[t_{23},\R_{1(23)4}]-[t_{(23)4},\R_{123}]=&~0&&,\label{eq:4th of 5 relations}\\
[t_{12},\R_{(12)34}]-[t_{34},\L_{12(34)}]=&~0&&.\label{eq:5th of 5 relations}
\end{alignat}
\end{subequations}
\end{lem}
\begin{proof}
We first prove \eqref{eq:1/5 relation}, for $U,V,W,X\in\C$,
\begin{align}\label{eq:2-deformed left 4T relation}
(\R_{UVW}\otimes1_X)t_{(UVW)X}-t_{(UVW)X}(\R_{UVW}\otimes1_X)\overset{\eqref{eq:dubindex is homotopy}}{=}\partial(t_{\R_{UVW},1_X})+t_{\partial(\R_{UVW}),1_X}
\end{align}
but the truncation annihilates $t_{\R_{UVW},1_X}$ thus we rewrite the RHS of \eqref{eq:2-deformed left 4T relation} as
\begin{align*}
t_{\partial(\R_{UVW}),1_X}\overset{\eqref{eq:deformed right 4T relation}}{=}\quad&t_{(1_U\otimes t_{VW})t_{U(VW)},1_X}-t_{t_{U(VW)}(1_U\otimes t_{VW}),1_X}\\
\overset{\eqref{eqn:dubindex splits prods}}{=}\quad\,&t_{1_U\otimes t_{VW},1_X}(t_{U(VW)}\otimes1_X)+(1_U\otimes t_{VW}\otimes1_X)t_{t_{U(VW)},1_X}\\&-t_{t_{U(VW)},1_X}(1_U\otimes t_{VW}\otimes1_X)-(t_{U(VW)}\otimes1_X)t_{1_U\otimes t_{VW},1_X}\\
\overset{\eqref{eq:t_(fg)h},\eqref{eq:L_123:=}}{=}&[1_U\otimes\L_{VWX},t_{U(VW)}\otimes1_X]-[\L_{U(VW)X},1_U\otimes t_{VW}\otimes1_X]\quad.\numberthis
\end{align*}
The proof of \eqref{eq:2/5 relation} is the same but uses the \eqref{eq:t_f(gh)} instead of \eqref{eq:t_(fg)h}; likewise, the proofs of \eqref{eq:3/5 relation} and \eqref{eq:4th of 5 relations} are the same but make use of the deformed left four-term relations \eqref{eq:deformed left 4T relation} instead of the deformed right four-term relations \eqref{eq:deformed right 4T relation}. Lastly, we prove \eqref{eq:5th of 5 relations}, 
\begin{align*}
t_{t_{UV},1_{WX}}(1_{UV}\otimes t_{WX})+(t_{UV}\otimes1_{WX})t_{1_{UV},t_{WX}}\overset{\eqref{eqn:dubindex splits prods}}{=}&t_{t_{UV},t_{WX}}\\
\overset{\eqref{eqn:dubindex splits prods}}{=}&t_{1_{UV},t_{WX}}(t_{UV}\otimes1_{WX})\\&+(1_{UV}\otimes t_{WX})t_{t_{UV},1_{WX}}\quad.\numberthis
\end{align*}
\end{proof}
We now reprove Cirio and Martins' result \cite[Theorems 21 and 22]{Joao} that coherent totally symmetric infinitesimal 2-braidings satisfy six categorified relations that replace the four-term relations \eqref{eq:four-term relations}. 
\begin{cor}
If the infinitesimal 2-braiding of Lemma \ref{lem:5 relations} is coherent and totally symmetric then we have the following six relations:
\begin{subequations}\label{subeq:again rewritten 5 relations}
\begin{alignat}{6}
[t_{(123)4},\R_{123}]-[t_{1(23)},\L_{234}]+[t_{23},\L_{124}+\L_{134}]=&~0&&,\label{eq:again 1st relation}\\
[t_{1(234)},\R_{234}]+[t_{34},\R_{123}+\R_{124}]-[t_{2(34)},\R_{134}]=&~0&&,\label{eq:again 2nd relation}\\
[t_{(123)4},\L_{123}]+[t_{12},\L_{134}+\L_{234}]-[t_{(12)3},\L_{124}]=&~0\quad&&,\label{eq:again 3rd relation}\\
[t_{1(234)},\L_{234}]+[t_{23},\R_{124}+\R_{134}]-[t_{(23)4},\R_{123}]=&~0&&,\label{eq:again 4th relation}\\
[t_{12},\R_{134}+\R_{234}]-[t_{34},\L_{123}+\L_{124}]=&~0&&,\label{eq:again 5th relation}\\
[t_{13},\R_{124}-\L_{234}-\R_{234}]+[t_{24},\L_{123}+\R_{123}-\L_{134}]=&~0&&.\label{eq:Theorem 22 of CM}
\end{alignat}
\end{subequations}
\end{cor}
\begin{proof}
As in \cite[Lemma 5.23]{Kem25a}, a totally symmetric infinitesimal 2-braiding $t$ gives us:
\begin{equation}\label{subeq:L is primitive}
\L_{12(34)}=\L_{123}+\L_{124}\quad,\quad\L_{1(23)4}=\L_{124}+\L_{134}\quad,\quad\L_{(12)34}=\L_{134}+\L_{234}\quad,
\end{equation}
and likewise for $\R$ . Thus the first 5 relations in \eqref{subeq:again rewritten 5 relations} come from the total symmetry of $t$. The last relation \eqref{eq:Theorem 22 of CM} comes from applying the permutation $(2\leftrightarrow3)$ to \eqref{eq:again 5th relation} and using the fact that $t$ is coherent.
\end{proof}
The above (together with `disjoint-commutativity', e.g. \cite[(5.18c)]{Kem25a}) motivates the following definition. 
\begin{defi}\label{def:nth DK 2-algebra}
For $n\in\bbN$, the \textbf{$n^\textbf{th}$ Drinfeld-Kohno 2-algebra} is the associative 2-algebra generated by
\begin{equation}
\Big\{a_{ij}\in A\,,\,\ell_{ijk}\in B\,,\,r_{ijk}\in B\,\Big|\,1\leq i<j<k\leq n+1\Big\}
\end{equation}
such that 
\begin{equation}
\partial\big(\ell_{ijk}\big)=\big[a_{ij},a_{ik}+a_{jk}\big]\qquad,\qquad\partial\big(r_{ijk}\big)=\big[a_{jk},a_{ij}+a_{ik}\big]
\end{equation}
and subject to the relations:
\begin{enumerate}
\item[(i)] For $1\leq i<j<k<l\leq n+1$,
\begin{subequations}
\begin{alignat}{6}
\big[a_{il}+a_{jl}+a_{kl},r_{ijk}\big]-\big[a_{ij}+a_{ik},\ell_{jkl}\big]+\big[a_{jk},\ell_{ijl}+\ell_{ikl}\big]=&~0&&,\\
\big[a_{ij}+a_{ik}+a_{il},r_{jkl}\big]+\big[a_{kl},r_{ijk}+r_{ijl}\big]-\big[a_{jk}+a_{jl},r_{ikl}\big]=&~0&&,\\
\big[a_{il}+a_{jl}+a_{kl},\ell_{ijk}\big]+\big[a_{ij},\ell_{ikl}+\ell_{jkl}\big]-\big[a_{ik}+a_{jk},\ell_{ijl}\big]=&~0\quad&&,\\
\big[a_{ij}+a_{ik}+a_{il},\ell_{jkl}\big]+\big[a_{jk},r_{ijl}+r_{ikl}\big]-\big[a_{jl}+a_{kl},r_{ijk}\big]=&~0&&,\\
\big[a_{ij},r_{ikl}+r_{jkl}\big]-\big[a_{kl},\ell_{ijk}+\ell_{ijl}\big]=&~0&&,\\
\big[a_{ik},r_{ijl}-\ell_{jkl}-r_{jkl}\big]+\big[a_{jl},\ell_{ijk}+r_{ijk}-\ell_{ikl}\big]=&~0&&.
\end{alignat}
\end{subequations}
\item[(ii)] If $\{1\leq i<j\leq n+1\}\cap\{1\leq k<l\leq n+1\}=\varnothing$ then
\begin{equation}
\big[a_{ij},a_{kl}\big]=0\quad.
\end{equation}
\item[(iii)] If $\{1\leq i<j\leq n+1\}\cap\{1\leq k<l<m\leq n+1\}=\varnothing$ then, for $b_{klm}\in\{\ell_{klm},r_{klm}\}$,
\begin{equation}
\big[a_{ij},b_{klm}\big]=0\quad.
\end{equation}
\end{enumerate} 
\end{defi}
We are now ready to state our fundamental conjecture in a very concise form.
\begin{conj}\label{conj:fundamental}
For $n\in\bbN$, the $n^\mathrm{th}$ Drinfeld-Kohno 2-algebra is acyclic, i.e. $\ker(\partial)=0$.
\end{conj}
In the context of Example \ref{ex:End_otimes^n}, if we are given a coherent totally symmetric infinitesimal 2-braiding $t$ then, by construction of the definition, we have an $n^\mathrm{th}$ Drinfeld-Kohno 2-algebra as the subalgebra of $\End_{\otimes^n}$ generated by:
\begin{equation}
t_{ij}:\otimes^n\Rightarrow\otimes^n\quad,\quad\L_{ijk}:[t_{ij},t_{ik}+t_{jk}]\Rrightarrow0\quad,\quad\R_{ijk}:[t_{jk},t_{ij}+t_{ik}]\Rrightarrow0\quad,
\end{equation}
where $1\leq i<j<k\leq n+1$. In this case, Conjecture \ref{conj:fundamental} states that every modification in the $n^\mathrm{th}$ Drinfeld-Kohno 2-algebra of the form $\Xi:0\Rrightarrow0$ vanishes. This conjecture seems somewhat obvious for low $n\in\bbN$; for instance, the $2^\mathrm{nd}$ Drinfeld-Kohno 2-algebra \cite[Example 3.31]{Kem25b} is the subalgebra of $\End_{\otimes^2}$ generated \textit{freely} by $t_{12},t_{23},t_{13}:\otimes^2\Rightarrow\otimes^2$ together with $\L:[t_{12},t_{13}+t_{23}]\Rrightarrow0$ and $\R:[t_{23},t_{12}+t_{13}]\Rrightarrow0$.

\begin{rem}\label{rem:conjecture is fundamental}
Let us explain the power of Conjecture \ref{conj:fundamental}. Given a coherent totally symmetric infinitesimal 2-braiding $t$, we can strip the four axioms \eqref{eq:left tetrahedron}-\eqref{eq:Breen axiom index version} of instances of the symmetric braiding $\gamma$ to reveal equations in terms of $R$ and $t$. For example, \cite[Construction 5.1]{Kem25b} showed that the Breen polytope axiom \eqref{eq:Breen axiom index version} reduces to
\begin{align*}
&\big(e^{i\pi\hbar t}\big)_{e^{i\pi\hbar t_{12}}}+\Phi(t_{23},t_{12})\left[R_{213}\Phi(t_{12},t_{13})e^{i\pi t_{12}}-e^{i\pi t_{23}}\Phi(t_{13},t_{23})R_{321}+\big(e^{i\pi\hbar t}\big)_{e^{i\pi\hbar t_{23}}}\Phi(t_{12},t_{23})\right]\\
&+\left[R_{231}\Phi(t_{23},t_{13})e^{i\pi t_{23}}-e^{i\pi t_{12}}\Phi(t_{13},t_{12})R\right]\Phi(t_{12},t_{23})=0\quad.\numberthis\label{Breen polytope in t only}
\end{align*}
We say that terms like
\begin{subequations}\label{subeq:left congruence}
\begin{equation}
\big(e^{i\pi\hbar t}\big)_{e^{i\pi\hbar t_{12}}}\overset{\eqref{eqn:compositioncoherences}}{=}*^2_{\Id_\otimes,e^{i\pi\hbar t}\smallbox\Id_{\id_\C}|e^{i\pi\hbar t},\Id_{\otimes\smallbox\id_\C}}:e^{i\pi\hbar t_{12}}e^{i\pi\hbar t_{(12)3}}\Rrightarrow e^{i\pi\hbar t_{(12)3}}e^{i\pi\hbar t_{12}}
\end{equation}
are \textbf{congruences}; their explicit series formula is straightforward to derive (see \cite[(5.14)]{Kem25b}),
\begin{equation}
\big(e^{i\pi\hbar t}\big)_{e^{i\pi\hbar t_{12}}}=\sum_{\begin{smallmatrix}j=1\\k=1
\end{smallmatrix}}^\infty\frac{(i\pi\hbar)^{j+k}}{j!k!}\sum_{\begin{smallmatrix}1\leq l\leq j\\1\leq m\leq k
\end{smallmatrix}}t_{(12)3}^{m-1}\,t_{12}^{l-1}\L\,t_{12}^{j-l}t_{(12)3}^{k-m}\quad.
\end{equation}
\end{subequations}
The LHS of \eqref{Breen polytope in t only} is a modification endomorphic on $e^{i\pi\hbar t_{12}}e^{i\pi\hbar t_{(12)3}}$ but, using the linearity of pseudonatural transformations and modifications, that is the same thing as a modification endomorphic on 0 thus (given a series formula for $R$ in terms of $\L,\,\R$ and whiskerings by $t$) an element of the $2^\mathrm{nd}$ Drinfeld-Kohno 2-algebra of the form $\Xi:0\Rrightarrow0$. The other axioms \eqref{eq:associahedron}-\eqref{eq:Hexahedron} likewise demand the vanishing of some endomorphic modification made up of $\L,\,\R$ and whiskerings by $t$. 
\end{rem}

%%%%%%%%%%%%%%%%%%%%%%%%%%%%%%%%%%%%%%%%%%%%%%%%%%%%%%%%%%%%%

\section{Algebraic construction of the hexagonator series}\label{subsec:hexagonator}
This section reproduces our direct algebraic construction \cite[(4.78)]{Kem25b} of the hexagonator series. In contrast to \cite[Section 4]{Kem25b} and Section \ref{subsec:pentagonator}, Theorem \ref{theo:hexagonator} does not make use of any higher gauge theoretic methods concerning 2-connections and their 2-holonomy \cite{Baez,On2dhol}. We must first recall the explicit formula \cite[Theorem A.9]{LeMu} for Drinfeld's Knizhnik-Zamolodchikov associator series so we begin with the notion of a multiple zeta value.
\begin{defi}\label{def:MZV}
If $k\in\bbN\setminus\{0\}$, $s_1\in\bbN\setminus\{0,1\}$ and $s_2,\ldots,s_k\in\bbN\setminus\{0\}$ then we call
\begin{equation}\label{eq:multiple zeta function}
\zeta(s_1,\ldots,s_k):=\sum_{n_1>n_2>\cdots>n_k\geq1}^\infty\frac{1}{n_1^{s_1}\cdots n_k^{s_k}}
\end{equation} 
a \textbf{multiple zeta value} (MZV) or \textbf{Euler sum}.
\end{defi}
Given a finite length non-empty tuple $p$ of natural numbers, we denote such length tautologically as $\p$ thus $p=(p_1,\ldots,p_{\p})$. In this case, $p>0$ means that every entry is strictly positive, i.e. $p_1,\ldots,p_{\p}\in\bbN\setminus\{0\}$. Given another tuple $j$ of natural numbers, by $0\leq j\leq p$ we mean that $\j=\p$ and $0\leq j_i\leq p_i$ for all $1\leq i\leq\p$. We define $|p|:=\sum_{l=1}^{\p}p_l$ and, for $q>0$ such that $\q=\p$,  
\begin{equation}\label{eq:compact zeta notation}
\zeta_j^{p,q}:=(-1)^{|j|+|p|}\zeta\left(p_1+1,\{1\}^{q_1-1},\ldots,p_{\p}+1,\{1\}^{q_{\p}-1}\right)\prod_{l=1}^{\p}\binom{p_l}{j_l}\quad.
\end{equation} 
\begin{defi}\label{def:Drinfeld KZ series}
Given elements $A$ and $B$ of an associative unital $\bbC$-algebra and a formal deformation parameter $\hbar$, \textbf{Drinfeld's Knizhnik-Zamolodchikov associator series} $\Phi(A,B)$ is the following element of $\bbC\langle A,B\rangle[[\hbar]]$,
\begin{equation}\label{eq:LM expression for Phi}
1+\sum_{\{p,q>0\,|\,\p=\q\}}\hbar^{|p|+|q|}\sum_{\begin{smallmatrix}0\leq j\leq p\\0\leq k\leq q\end{smallmatrix}}\zeta_j^{p,q}\left(\prod_{l=1}^{\p}\binom{q_l}{k_l}(-1)^{k_l}\right)B^{|q|-|k|}A^{j_1}B^{k_1}\cdots A^{j_{\p}}B^{k_{\p}}A^{|p|-|j|}\quad.
\end{equation}
\end{defi}
\begin{rem}
We can compactify the expression for Drinfeld's Knizhnik-Zamolodchikov associator series \eqref{eq:LM expression for Phi} in two different ways, both of which we will need:
\begin{enumerate}
\item[(i)] We set $j_0:=0\,,~k_0:=|q|-|k|\,,~j_{\p+1}:=|p|-|j|\,,~k_{\p+1}:=0$ and
\begin{equation}
\zeta_{j,k}^{p,q}\,:=\,\zeta_j^{p,q}\prod_{l=1}^{\p}\binom{q_l}{k_l}(-1)^{k_l}
\end{equation}
so that \eqref{eq:LM expression for Phi} equals
\begin{equation}\label{eq:most compact expression of Drinfeld KZ series}
\Phi(A,B)\,=\,1+\sum_{\{p,q>0\,|\,\p=\q\}}\hbar^{|p|+|q|}\sum_{\begin{smallmatrix}0\leq j\leq p\\0\leq k\leq q\end{smallmatrix}}\zeta_{j,k}^{p,q}\prod_{l=0}^{\p+1}A^{j_l}B^{k_l}\quad.
\end{equation}
\item[(ii)] Let $\r_A$ denote right-multiplication by $A$ then \eqref{eq:LM expression for Phi} equals
\begin{equation}\label{eq:2nd most compact expression of Drinfeld KZ series}
\Phi(A,B)\,=\,1+\sum_{\{p,q>0\,|\,\p=\q\}}\hbar^{|p|+|q|}\sum_{0\leq j\leq p}\zeta_j^{p,q}\left(\ad^{q_{\p}}_B\r_A^{j_{\p}}\cdots\ad^{q_1}_B\r_A^{j_1}(1)\right)A^{|p|-|j|}\quad.
\end{equation}
\end{enumerate}
\end{rem}
We recall the \textbf{BRW identity} in $\bbC\langle A,B\rangle[[\hbar]]$ between Drinfeld's KZ associator series and the exponential \cite[Last equation in the proof of Theorem 22]{BRW}, i.e.
\begin{equation}\label{eq:BRW functional relation with A and B}
\Phi(A,-A-B)e^{-i\pi\hbar A}\Phi(B,A)= e^{-i\pi\hbar(A+B)}\Phi(B,-A-B)e^{i\pi\hbar B}\quad.
\end{equation}
For an infinitesimal 2-braiding $t$, we set $\Lambda:=t_{12}+t_{23}+t_{13}$ and $\t:=t_{13}-\Lambda$. We substitute $A=t_{12}$ and $B=t_{23}$ into \eqref{eq:BRW functional relation with A and B} while absorbing factors of $\hbar$ into $t$,
\begin{equation}\label{eq:BRW functional relation}
\Phi(t_{12},\t)e^{-i\pi t_{12}}\Phi(t_{23},t_{12})= e^{i\pi\t}\Phi(t_{23},\t)e^{i\pi t_{23}}\quad.
\end{equation}
\begin{theo}\label{theo:hexagonator}
We have an explicit formula for the right pre-hexagonator series
\begin{equation}\label{eq:theo hexagonator}
R:\Phi(t_{12},t_{13})e^{i\pi t_{(12)3}}\Phi(t_{23},t_{12})\xRrightarrow{~~~} e^{i\pi\hbar t_{13}}\Phi(t_{23},t_{13})e^{i\pi\hbar t_{23}}\quad.
\end{equation}
\end{theo}
\begin{proof}
If we have explicit series formulae for the following modifications:
\begin{subequations}
\begin{alignat}{6}
\int_{e^{i\pi t_{(12)3}}}^{e^{i\pi\Lambda}e^{-i\pi t_{12}}}&:\quad~e^{i\pi t_{(12)3}}&&\xRrightarrow{~~~}\quad e^{i\pi\Lambda}e^{-i\pi t_{12}}&&&,\label{eq:1/5 modification}\\
\int_{\Phi(t_{12},t_{13})e^{i\pi\Lambda}}^{e^{i\pi\Lambda}\Phi(t_{12},t_{13})}&:\Phi(t_{12},t_{13})e^{i\pi\Lambda}&&\xRrightarrow{~~~} e^{i\pi\Lambda}\Phi(t_{12},t_{13})\quad&&&,\label{eq:2/5 modification}\\
\int_{\Phi(t_{12},t_{13})}^{\Phi(t_{12},\t)}&:\quad\Phi(t_{12},t_{13})&&\xRrightarrow{~~~}\quad\Phi(t_{12},\t)&&&,\label{eq:3/5 modification}\\
\int_{\Phi(t_{23},\t)}^{\Phi(t_{23},t_{13})}&:\quad\Phi(t_{23},\t)&&\xRrightarrow{~~~}\quad\Phi(t_{23},t_{13})&&&,\label{eq:4/5 modification}\\
\int_{e^{i\pi\Lambda}e^{i\pi\t}}^{e^{i\pi t_{13}}}&:\quad e^{i\pi\Lambda}e^{i\pi\t}&&\xRrightarrow{~~~}\qquad e^{i\pi t_{13}}&&&,\label{eq:5/5 modification}
\end{alignat}
\end{subequations}
then we can use \eqref{eq:BRW functional relation} to construct \eqref{eq:theo hexagonator} as 
\begin{align*}
R:=&\left(\Phi(t_{12},t_{13})\int_{e^{i\pi t_{(12)3}}}^{e^{i\pi\Lambda}e^{-i\pi t_{12}}}+\int_{\Phi(t_{12},t_{13})e^{i\pi\Lambda}}^{e^{i\pi\Lambda}\Phi(t_{12},t_{13})}e^{-i\pi t_{12}}+e^{i\pi\Lambda}\int_{\Phi(t_{12},t_{13})}^{\Phi(t_{12},\t)}e^{-i\pi t_{12}}\right)\Phi(t_{23},t_{12})\\
&~+e^{i\pi\Lambda}e^{i\pi\t}\int_{\Phi(t_{23},\t)}^{\Phi(t_{23},t_{13})}e^{i\pi t_{23}}+\int_{e^{i\pi\Lambda}e^{i\pi\t}}^{e^{i\pi t_{13}}}\Phi(t_{23},t_{13})e^{i\pi t_{23}}\numberthis\quad.
\end{align*}
The modifications \eqref{eq:1/5 modification} and \eqref{eq:5/5 modification} were explicitly determined in \cite[(4.57b) and (4.73), respectively]{Kem25b},
\begin{alignat}{6}
&\int_{e^{i\pi t_{(12)3}}}^{e^{i\pi\Lambda}e^{-i\pi t_{12}}}=&&~\sum_{k=2}^\infty\frac{(i\pi)^k}{k!}\sum_{l=1}^{k-1}\sum_{m=0}^{k-l-1}\sum_{n=0}^{k-l-m-1}\binom{k-l}{m}(-1)^{m+1}t_{(12)3}^{l-1}\Lambda^n\L\Lambda^{k-l-m-n-1}t_{12}^m~&&&,\label{eq:BCH mod decomposing e^t_(12)3}\\
&\int_{e^{i\pi\Lambda}e^{i\pi\t}}^{e^{i\pi t_{13}}}~~\,=&&~\sum_{k=2}^\infty\frac{(i\pi)^k}{k!}\sum_{l=1}^{k-1}\sum_{m=0}^{k-l-1}\sum_{n=0}^{k-l-m-1}\binom{k-l}{m}t_{13}^{l-1}\Lambda^n(\L+\R)\Lambda^{k-l-m-n-1}\t^m&&&.
\end{alignat}
The modification $\int_{\Phi(t_{12},t_{13})e^{i\pi\Lambda}}^{e^{i\pi\Lambda}\Phi(t_{12},t_{13})}$ uses the alternative expression \eqref{eq:most compact expression of Drinfeld KZ series} for Drinfeld's KZ series and was determined in \cite[(4.71)]{Kem25b}, it is given as
\begin{align}
&\sum_{\begin{smallmatrix}
\{p,q>0\,|\,\p=\q\}\\1\leq m<\infty
\end{smallmatrix}}\sum_{\begin{smallmatrix}0\leq j\leq p\\0\leq k\leq q\end{smallmatrix}}\sum_{\begin{smallmatrix}0\leq l\leq\p+1\\1\leq n\leq m\end{smallmatrix}}\frac{(i\pi)^m}{m!}\zeta_{j,k}^{p,q}\Lambda^{n-1}\label{eq:mod commuting Phi_213 with e^lambda}\\
&\qquad\qquad~\times\left(\prod_{r=0}^{l-1}t_{12}^{j_r}t_{13}^{k_r}\right)\left(\sum_{r=1}^{j_l}t_{12}^{r-1}\L t_{12}^{j_l-r}t_{13}^{k_l}-t_{12}^{j_l}\sum_{r=1}^{k_l}t_{13}^{r-1}(\L+\R)t_{13}^{k_l-r}\right)\left(\prod_{r=l+1}^{\p+1}t_{12}^{j_r}t_{13}^{k_r}\right)\Lambda^{m-n}\nn
\end{align}
The modification $\int_{\Phi(t_{12},t_{13})}^{\Phi(t_{12},\t)}$ uses the other alternative expression \eqref{eq:2nd most compact expression of Drinfeld KZ series} for Drinfeld's KZ series and was determined in \cite[(4.29e)]{Kem25b},
\begin{align}\label{eq:debarring mod}
\nn&\sum_{\{p,q>0\,|\,\p=\q\}}\sum_{\begin{smallmatrix}0\leq j\leq p\\1\leq l\leq\p 
\end{smallmatrix}}\zeta_j^{p,q}\ad^{q_{\p}}_{t_{13}}\r_{t_{12}}^{j_{\p}}\cdots\ad^{q_{l+1}}_{t_{13}}\r_{t_{12}}^{j_{l+1}}\sum_{m=0}^{q_l-1}\ad_{t_{13}}^{q_l-m-1}\Bigg(\sum_{k_1=0}^{q_1}(-1)^{k_1}\binom{q_1}{k_1}\cdots\\&\cdots\sum_{k_{l-1}=0}^{q_{l-1}}(-1)^{k_{l-1}}\binom{q_{l-1}}{k_{l-1}}\sum_{k_l=0}^m(-1)^{k_l}\binom{m}{k_l}\sum_{n=0}^l\left(\prod_{r=0}^{n-1}t_{12}^{j_r}\t^{k_r}\right)\Bigg[t_{12}^{j_n}\sum_{r=1}^{k_n}\t^{r-1}(\L+\R)\t^{k_n-r}\nn\\
&-\sum_{r=1}^{j_n}t_{12}^{r-1}\L t_{12}^{j_n-r}\t^{k_n}\Bigg]t_{12}^{j_{n+1}}\t^{k_{n+1}}\cdots t_{12}^{j_l}\t^{k_l}\Bigg)t_{12}^{|p|-|j|}
\end{align}
where $j_0:=0$ and $k_0:=m-k_l+\sum_{n=1}^{l-1}(q_n-k_n)$. The modification $\int_{\Phi(t_{23},\t)}^{\Phi(t_{23},t_{13})}$ can be acquired from \eqref{eq:debarring mod} by multiplying by $-1$ and applying the index permutation $(1\leftrightarrow3)$.
\end{proof}
\section{Construction of the pentagonator series}\label{subsec:pentagonator}
Choosing $n=3$ in \cite[Definition 3.27]{Kem25b}, we express the 2-connection $\big(\A_\KZ^{n=3},\B_\CM^{n=3}\big)$ on $Y_4$:
\begin{subequations}\label{subeq:KZ 2-connection on Y_4}
\begin{alignat*}{6}
\A_\KZ^{n=3}=&\left(\frac{dz_1-dz_2}{z_1-z_2}\right)t_{12}+\left(\frac{dz_1-dz_3}{z_1-z_3}\right)t_{13}+\left(\frac{dz_1-dz_4}{z_1-z_4}\right)t_{14}\\&+\left(\frac{dz_2-dz_3}{z_2-z_3}\right)t_{23}+\left(\frac{dz_2-dz_4}{z_2-z_4}\right)t_{24}+\left(\frac{dz_3-dz_4}{z_3-z_4}\right)t_{34}\numberthis\\
\end{alignat*}
and
\begin{alignat*}{6}
\B_\CM^{n=3}=&~\frac{2}{(z_3-z_1)}\left(\frac{\R_{123}}{z_2-z_3}-\frac{\L_{123}}{z_1-z_2}\right)(dz_1\wedge dz_2+dz_2\wedge dz_3+dz_3\wedge dz_1)\\&+\frac{2}{(z_4-z_1)}\left(\frac{\R_{124}}{z_2-z_4}-\frac{\L_{124}}{z_1-z_2}\right)(dz_1\wedge dz_2+dz_2\wedge dz_4+dz_4\wedge dz_1)\\&+\frac{2}{(z_4-z_1)}\left(\frac{\R_{134}}{z_3-z_4}-\frac{\L_{134}}{z_1-z_3}\right)(dz_1\wedge dz_3+dz_3\wedge dz_4+dz_4\wedge dz_1)\\&+\frac{2}{(z_4-z_2)}\left(\frac{\R_{234}}{z_3-z_4}-\frac{\L_{234}}{z_2-z_3}\right)(dz_2\wedge dz_3+dz_3\wedge dz_4+dz_4\wedge dz_2)\quad.\numberthis
\end{alignat*}
\end{subequations}
The \textbf{diagonally-punctured complex plane} is defined as
\begin{equation}\label{eq:diagonally punctured plane}
\bbC^2_\#:=\big\{(z,u)\in\bbC^2\,|\,zu(z-1)(u-1)(z-u)\neq0\big\}\qquad.
\end{equation}
The map 
\begin{equation}
\varphi:\bbC^2_\#\times\bbC^\times\times\bbC\longrightarrow Y_4\qquad,\qquad(z,u,v,w)\longmapsto(w,zv+w,uv+w,v+w)
\end{equation}
is a birational biholomorphism with inverse given by
\begin{equation}
\varphi^{-1}:Y_4\longrightarrow\bbC^2_\#\times\bbC^\times\times\bbC\qquad,\qquad(z_1,z_2,z_3,z_4)\longmapsto\left(\frac{z_2-z_1}{z_4-z_1},\frac{z_3-z_1}{z_4-z_1},z_4-z_1,z_1\right)~.
\end{equation}
We pullback the 2-connection $\big(\A_\KZ^{n=3},\B_\CM^{n=3}\big)$ of \eqref{subeq:KZ 2-connection on Y_4} along the birational biholomorphism $\varphi$ and define $\big(\A:=\varphi^*\A_\KZ^{n=3},\B:=\varphi^*\B_\CM^{n=3}\big)$ thus:
\begin{subequations}\label{subeq:pulled back 2-connection}
\begin{alignat*}{6}
\A=&\left(\frac{t_{12}}{z}+\frac{t_{23}}{z-u}+\frac{t_{24}}{z-1}\right)dz+\left(\frac{t_{13}}{u}+\frac{t_{23}}{u-z}+\frac{t_{34}}{u-1}\right)du\\&+\frac{t_{12}+t_{13}+t_{14}+t_{23}+t_{24}+t_{34}}{v}dv\numberthis\\
\end{alignat*}
and
\begin{alignat*}{6}
\B=&~2\left(\frac{\L_{123}}{zu}+\frac{\R_{123}}{u(z-u)}+\frac{\L_{234}}{(1-z)(u-z)}+\frac{\R_{234}}{(1-z)(u-1)}\right)dz\wedge du\\&+\frac{2}{v}\left(\frac{\R_{123}+\L_{234}}{u-z}-\frac{\L_{123}+\L_{124}}{z}+\frac{\R_{124}-\L_{234}-\R_{234}}{1-z}\right)dv\wedge dz\\&+\frac{2}{v}\left(\frac{\L_{134}-\L_{123}-\R_{123}}{u}+\frac{\L_{234}+\R_{123}}{u-z}+\frac{\R_{134}+\R_{234}}{u-1}\right)du\wedge dv\quad.\numberthis
\end{alignat*}
\end{subequations}
As in \cite[Remark 3.28]{Kem25b}, this 2-connection is automatically fake flat; Cirio and Martins constructed it such \cite{Joao1,Joao,Joao2}. We now recontextualise \cite[Theorem 23]{Joao} and demonstrate a sufficient condition for the 2-flatness of \eqref{subeq:pulled back 2-connection}.
\begin{propo}
If $t$ is coherent and totally symmetric then \eqref{subeq:pulled back 2-connection} is 2-flat.
\end{propo}
\begin{proof}
We define the modification $M$ as
\begin{equation}
\A\wedge^{[\cdot,\cdot]}\B=\frac{2^3}{v}Mdz\wedge du\wedge dv
\end{equation}
thus
\begin{align*}
M:=&~\frac{1}{zu}\left(\big[t_{12},\L_{134}-\R_{123}\big]-\big[t_{13},\L_{124}\big]+\big[t_{(13)4}+t_{2(34)},\L_{123}\big]\right)\\
&+\frac{1}{z(u-1)}\left(\big[t_{12},\R_{134}+\R_{234}\big]-\big[t_{34},\L_{123}+\L_{124}\big]\right)\\&+\frac{1}{z(u-z)}\left(\big[t_{12},\L_{234}+\R_{123}\big]-\big[t_{23},\L_{123}+\L_{124}\big]\right)\\
&+\frac{1}{u(u-z)}\left(\big[t_{23},\L_{123}-\L_{134}\big]+\big[t_{13},\L_{234}\big]-\big[t_{1(24)}+t_{(23)4},\R_{123}\big]\right)\\&+\frac{1}{(u-z)(u-1)}\left(\big[t_{34},\R_{123}+\L_{234}\big]-\big[t_{23},\R_{134}+\R_{234}\big]\right)\\&+\frac{1}{u(1-z)}\left(\big[t_{24},\L_{123}+\R_{123}-\L_{134}\big]+\big[t_{13},\R_{124}-\L_{234}-\R_{234}\big]\right)\\&+\frac{1}{(u-z)(1-z)}\left(\big[t_{23},\R_{124}-\R_{234}\big]-\big[t_{24},\R_{123}\big]+\big[t_{1(23)}+t_{(13)4},\L_{234}\big]\right)\\&+\frac{1}{(u-1)(1-z)}\left(\big[t_{34},\R_{124}-\L_{234}\big]-\big[t_{24},\R_{134}\big]+\big[t_{(12)3}+t_{1(24)},\R_{234}\big]\right)\qquad.\numberthis\label{eq:M as defined}
\end{align*}
It is straightforward to check that \eqref{eq:M as defined} simplifies to 
\begin{align*}
M=&~\frac{1}{zu}\left(\big[t_{12},\L_{134}+\L_{234}\big]-\big[t_{(12)3},\L_{124}\big]+\big[t_{(123)4},\L_{123}\big]\right)\\&+\frac{1}{z(u-1)}\left(\big[t_{12},\R_{134}+\R_{234}\big]-\big[t_{34},\L_{123}+\L_{124}\big]\right)\\&+\frac{1}{u(z-u)}\left(\big[t_{23},\L_{124}+\L_{134}\big]-\big[t_{1(23)},\L_{234}\big]+\big[t_{(123)4},\R_{123}\big]\right)\\&+\frac{1}{u(1-z)}\left(\big[t_{24},\L_{123}+\R_{123}-\L_{134}\big]+\big[t_{13},\R_{124}-\L_{234}-\R_{234}\big]\right)\\&+\frac{1}{(u-z)(1-z)}\left(\big[t_{23},\R_{124}+\R_{134}\big]-\big[t_{(23)4},\R_{123}\big]+\big[t_{1(234)},\L_{234}\big]\right)\\&+\frac{1}{(u-1)(1-z)}\left(\big[t_{34},\R_{123}+\R_{124}\big]-\big[t_{2(34)},\R_{134}\big]+\big[t_{1(234)},\R_{234}\big]\right)\numberthis
\end{align*}
which vanishes upon using \eqref{subeq:again rewritten 5 relations}.
\end{proof}
Following \cite[Subsection 2.4]{BRW}, we restrict to the following open triangle in $\bbR^2$,
\begin{equation}\label{eq:open full triangle}
U':=\{w=0<x=z<y=u<1=v\}\subset\bbC^2_\#\hookrightarrow\bbC^2_\#\times\bbC^\times\times\bbC\quad.
\end{equation}
\begin{rem}\label{rem:2-connection on open triangle}
Restricting to this subspace simplifies the pullback 2-connection \eqref{subeq:pulled back 2-connection} as follows:
\begin{alignat}{6}
\A_{|U'}=&\left(\frac{t_{12}}{x}+\frac{t_{23}}{x-y}+\frac{t_{24}}{x-1}\right)dx+\left(\frac{t_{13}}{y}+\frac{t_{23}}{y-x}+\frac{t_{34}}{y-1}\right)dy&&,\label{eq:connection 1-form pulled back and restricted}\\
\B_{|U'}=&~2\left(\frac{\L_{123}}{xy}+\frac{\R_{123}}{y(x-y)}+\frac{\L_{234}}{(1-x)(y-x)}+\frac{\R_{234}}{(1-x)(y-1)}\right)dx\wedge dy\quad&&.
\end{alignat}
Importantly, this 2-connection is still \textit{not} flat on the nose but only fake flat hence we will need to construct a 2-path whose 2-holonomy will contribute to the pentagonator series.
\end{rem}
We make use of Bordemann, Rivezzi and Weigel's affine 1-paths \cite[Figure 2]{BRW}:
\begin{subequations}
\begin{alignat}{6}
c_\i(r):=&\left((1-r)\v^2+r(\v-\v^2),\v\right)&&,\label{eq:c_I(r):=}\\
c_\ii(r):=&\,(1-r)\big(\v-\v^2,\v\big)+r\big(1-\v,1-\v+\v^2\big)&&,\\
c_\iii(r):=&\left(1-\v,(1-r)(1-\v+\v^2)+r(1-\v^2)\right)\qquad&&,\label{eq:c_III(r):=}\\
c_\iv(r):=&\left(\v^2,(1-r)\v+r(1-\v^2)\right)&&,\\
c_\V(r):=&\left((1-r)\v^2+r(1-\v),1-\v^2\right)&&.
\end{alignat}
\end{subequations}
Setting 
\begin{equation}\label{eq:c^s_II(r):=}
c^s_\ii(r):=(1-r)\,c_\i(1-s)+r\,c_\iii(s)\quad,
\end{equation}
we define a 2-path $c_\ii\,c_\i\xRightarrow{P_\i}(c_\iii\circ\iota)\,c^1_\ii$ as
\begin{equation}\label{eq:P_I 2-path}
P_\i(s,r):=\begin{cases}
      c_\i(2r)\,, & 0\leq r\leq\frac{1-s}{2}\\
      c^s_\ii(2r+s-1)\,, & \frac{1-s}{2}\leq r\leq1-\frac{s}{2}\\
      (c_\iii\circ\iota)(2r-1)\,, & 1-\frac{s}{2}\leq r\leq1
    \end{cases}\qquad.
\end{equation}
As in \cite[(4.10a)]{Kem25b}, the 2-holonomy of \eqref{eq:P_I 2-path} is given as
\begin{equation}\label{1st 2-holonomy}
W^{P_\i}=\int_0^1\int_{\frac{1-s}{2}}^{1-\frac{s}{2}}W_{1r}^{P_\i^s}\B\left[\frac{\partial P_\i^s}{\partial s},\frac{\partial P_\i^s}{\partial r}\right]W_{r0}^{P_\i^s}\,drds\quad.
\end{equation}
For $\frac{1-s}{2}\leq r\leq1-\frac{s}{2}$, we have
\begin{align*}
P_\i^s(r)\overset{\eqref{eq:P_I 2-path}}{=}\quad~\,&\,c^s_\ii(2r+s-1)\\
\overset{\eqref{eq:c^s_II(r):=}}{=}\quad~\,&\,(2-2r-s)\,c_\i(1-s)+(2r+s-1)\,c_\iii(s)\\
\overset{\eqref{eq:c_I(r):=},\eqref{eq:c_III(r):=}}{=}&\,\Big(\big(2-2r-s\big)\big(s\v^2+(1-s)(\v-\v^2)\big)+(2r+s-1)(1-\v)\\&\qquad,\,(2-2r-s)\v+\big(2r+s-1\big)\big((1-s)(1-\v+\v^2)+s(1-\v^2)\big)\Big)\\
=:\quad~~\,&\big(\,x(s,r)\,,\,y(s,r)\,\big)\numberthis\label{eq:P_I^s(r):=}
\end{align*}
which can be substituted in the expression for $\B\left[\frac{\partial P_\i^s}{\partial s},\frac{\partial P_\i^s}{\partial r}\right]$ given by
\begin{equation}
2\left(\frac{\L_{123}}{xy}+\frac{\R_{123}}{y(x-y)}+\frac{\L_{234}}{(1-x)(y-x)}+\frac{\R_{234}}{(1-x)(y-1)}\right)\left(\frac{\partial x}{\partial s}\frac{\partial y}{\partial r}-\frac{\partial x}{\partial r}\frac{\partial y}{\partial s}\right)\quad.
\end{equation}
Similarly, one has explicit expressions for the parallel transport terms $W_{1r}^{P_\i^s}$ and $W_{r0}^{P_\i^s}$ by evaluating the path-ordered exponential with respect to the connection \eqref{eq:connection 1-form pulled back and restricted} over the 1-path \eqref{eq:P_I^s(r):=}. 
Setting $c_\V^s(r):=\left((1-r)\v^2+r(1-\v),(1-s)\v+s(1-\v^2)\right)$, we define a 2-path $c^1_\ii\xRightarrow{P_\ii}c_\V\,c_\iv$,
\begin{equation}\label{eq:P_II 2-path}
P_\ii(s,r):=\begin{cases}
      c_\iv(2r)\,, & 0\leq r\leq\frac{s}{2}\\
      c_\V^s(2r-s)\,, & \frac{s}{2}\leq r\leq s\\
      c^1_\ii(r)\,, & s\leq r\leq1
    \end{cases}\qquad.
\end{equation}
As above, one has an explicit expression for the 2-holonomy
\begin{equation}\label{2nd 2-holonomy}
W^{P_\ii}=\int_0^1\int_{\frac{s}{2}}^sW_{1r}^{P_\ii^s}\B\left[\frac{\partial P_\ii^s}{\partial s},\frac{\partial P_\ii^s}{\partial r}\right]W_{r0}^{P_\ii^s}\,drds\quad.
\end{equation}
We define a 2-path $c_\iii\,c_\ii\,c_\i\xRightarrow{P}c_\V\,c_\iv$ as 
\begin{equation}
\begin{tikzcd}
	{c_\iii\,c_\ii\,c_\i} && {c_\V\,c_\iv} \\
	\\
	{c_\iii\,(c_\iii\circ\iota)\,c^1_\ii} && {c^1_\ii}
	\arrow["P", Rightarrow, dashed, from=1-1, to=1-3]
	\arrow["{c_\iii\,P_\i}"', Rightarrow, from=1-1, to=3-1]
	\arrow["{P_\mathrm{Triv}}"', Rightarrow, from=3-1, to=3-3]
	\arrow["{P_\ii}"', Rightarrow, from=3-3, to=1-3]
\end{tikzcd}\quad.
\end{equation}
The 2-functoriality of 2-holonomy \cite[Definition 3.25]{Kem25b} gives
\begin{equation}
W^P=W^{c_\iii}W^{P_\i}+W^{P_\ii}
\end{equation}
while the globularity condition imposes
\begin{align}
W^P:W^{c_\iii}W^{c_\ii}W^{c_\i}\xRrightarrow{~~~} W^{c_\V}W^{c_\iv}\quad.
\end{align}
\begin{theo}\label{theo:pentagonator}
Denoting $\Phi_{ijk}:=\Phi(t_{ij},t_{jk})$, we have an explicit formula for the pentagonator series 
\begin{equation}
\Pi:\Phi_{234}\Phi_{1(23)4}\Phi_{123}\Rrightarrow\Phi_{12(34)}\Phi_{(12)34}\quad.
\end{equation}
\end{theo}
\begin{proof}
We suppress the third argument in the LHS of \cite[(2.55)]{BRW} given that we are only actually interested in the limit $\v\to0$ and \cite[(2.62)]{BRW} guarantees that such harmless terms remain just that. With this point in mind, \cite[(2.63) and (2.64)]{BRW} gives us
\begin{align}\label{eq:W^P}
\v^{t_{34}}\Phi_{234}\v^{-t_{23}}\v^{t_{(23)4}}\Phi_{1(23)4}\v^{-t_{1(23)}}\v^{t_{23}}\Phi_{123}\v^{-t_{12}}\xRrightarrow{W^P} \v^{t_{2(34)}}\Phi_{12(34)}\v^{-2t_{12}}\v^{2t_{34}}\Phi_{(12)34}\v^{-t_{(12)3}}
\end{align}
which we rearrange as
\begin{align}\label{eq:conjugated W^P}
\Phi_{234}\v^{-t_{23}}\v^{t_{(23)4}}\Phi_{1(23)4}\v^{-t_{1(23)}}\v^{t_{23}}\Phi_{123}\xRrightarrow{M_0}\v^{-t_{34}} \v^{t_{2(34)}}\Phi_{12(34)}\v^{-2t_{12}}\v^{2t_{34}}\Phi_{(12)34}\v^{-t_{(12)3}}\v^{t_{12}}\,,
\end{align}
where $M_0:=\v^{-t_{34}}W^P\v^{t_{12}}$. By direct comparison with \eqref{subeq:left congruence}, we have:
\begin{subequations}\label{subeq:R_4 and L_1}
\begin{align}
\int_{\v^{t_{23}}\v^{-t_{1(23)}}}^{\v^{-t_{1(23)}}\v^{t_{23}}}=&\sum_{\begin{smallmatrix}j=1\\k=1
\end{smallmatrix}}^\infty(-1)^k\frac{(\ln\v)^{j+k}}{j!k!}\sum_{\begin{smallmatrix}1\leq l\leq j\\1\leq m\leq k
\end{smallmatrix}}t_{1(23)}^{m-1}t_{23}^{l-1}\R_{123}t_{23}^{j-l}t_{1(23)}^{k-m}:\v^{t_{23}}\v^{-t_{1(23)}}\Rrightarrow\v^{-t_{1(23)}}\v^{t_{23}}\,,\\
\int^{\v^{-t_{23}}\v^{t_{(23)4}}}_{\v^{t_{(23)4}}\v^{-t_{23}}}=&\sum_{\begin{smallmatrix}j=1\\k=1
\end{smallmatrix}}^\infty(-1)^{j+1}\frac{(\ln\v)^{j+k}}{j!k!}\sum_{\begin{smallmatrix}1\leq l\leq j\\1\leq m\leq k
\end{smallmatrix}}t_{(23)4}^{m-1}t_{23}^{l-1}\L_{234}t_{23}^{j-l}t_{(23)4}^{k-m}:\v^{t_{(23)4}}\v^{-t_{23}}\Rrightarrow\v^{-t_{23}}\v^{t_{(23)4}}\,,
\end{align}
\end{subequations}
which we can compose with \eqref{eq:conjugated W^P} to give
\begin{align}
\Phi_{234}\v^{t_{(23)4}}\v^{-t_{23}}\Phi_{1(23)4}\v^{t_{23}}\v^{-t_{1(23)}}\Phi_{123}\xRrightarrow{M_1}\v^{-t_{34}}\v^{t_{2(34)}}\Phi_{12(34)}\v^{-2t_{12}}\v^{2t_{34}}\Phi_{(12)34}\v^{-t_{(12)3}}\v^{t_{12}}\label{eq:1st move}
\end{align}
where 
\begin{equation}
M_1:=\Phi_{234}\left(\int^{\v^{-t_{23}}\v^{t_{(23)4}}}_{\v^{t_{(23)4}}\v^{-t_{23}}}\Phi_{1(23)4}\v^{t_{23}}\v^{-t_{1(23)}}+\v^{-t_{23}}\v^{t_{(23)4}}\Phi_{1(23)4}\int_{\v^{t_{23}}\v^{-t_{1(23)}}}^{\v^{-t_{1(23)}}\v^{t_{23}}}\right)\Phi_{123}+M_0\,.
\end{equation}
The modifications:
\begin{equation}\label{subeq:BCH mods for (23)4 and 1(23)}
\int_{\v^{\Lambda_{234}}\v^{-t_{23}}}^{\v^{t_{(23)4}}}:\v^{\Lambda_{234}}\v^{-t_{23}}\xRrightarrow{~~~}\v^{t_{(23)4}}\qquad,\qquad\int_{\v^{t_{23}}\v^{-\Lambda_{123}}}^{\v^{-t_{1(23)}}}:\v^{t_{23}}\v^{-\Lambda_{123}}\xRrightarrow{~~~}\v^{-t_{1(23)}}
\end{equation}
are analogous to \eqref{eq:BCH mod decomposing e^t_(12)3} and are given as, respectively:
\begin{subequations}\label{subeq:formula for BCH mods for (23)4 and 1(23)}
\begin{alignat}{6}
&\sum_{k=2}^\infty\frac{(\ln\v)^k}{k!}\sum_{l=1}^{k-1}\sum_{m=0}^{k-l-1}\sum_{n=0}^{k-l-m-1}\binom{k-l}{m}(-1)^mt_{(23)4}^{l-1}\Lambda_{234}^n\L_{234}\Lambda_{234}^{k-l-m-n-1}t_{23}^m&&,\\
&\sum_{k=2}^\infty\frac{(\ln\v)^k}{k!}\sum_{l=1}^{k-1}\sum_{m=0}^{k-l-1}\sum_{n=0}^{k-l-m-1}\binom{k-l}{m}(-1)^{l+m}t_{1(23)}^{l-1}t_{23}^n\R_{123}t_{23}^{k-l-m-n-1}\Lambda_{123}^m\quad&&.
\end{alignat}
\end{subequations}
Composing \eqref{subeq:BCH mods for (23)4 and 1(23)} with \eqref{eq:1st move}, we have
\begin{align}
\Phi_{234}\v^{\Lambda_{234}}\v^{-2t_{23}}\Phi_{1(23)4}\v^{2t_{23}}\v^{-\Lambda_{123}}\Phi_{123}\xRrightarrow{M_2}\v^{-t_{34}}\v^{t_{2(34)}}\Phi_{12(34)}\v^{-2t_{12}}\v^{2t_{34}}\Phi_{(12)34}\v^{-t_{(12)3}}\v^{t_{12}}\label{eq:2nd move}
\end{align}
where 
\begin{equation}
M_2:=\Phi_{234}\left(\v^{\Lambda_{234}}\v^{-2t_{23}}\Phi_{1(23)4}\v^{t_{23}}\int_{\v^{t_{23}}\v^{-\Lambda_{123}}}^{\v^{-t_{1(23)}}}+\int_{\v^{\Lambda_{234}}\v^{-t_{23}}}^{\v^{t_{(23)4}}}\v^{-t_{23}}\Phi_{1(23)4}\v^{t_{23}}\v^{-t_{1(23)}}\right)\Phi_{123}+M_1\,.
\end{equation}
The modifications:
\begin{subequations}\label{subeq:mods commuting Phi_234, Phi_1(23)4 and Phi_123}
\begin{alignat}{6}
&\int_{\Phi_{123}\v^{-\Lambda_{123}}}^{\v^{-\Lambda_{123}}\Phi_{123}}&&:\,\Phi_{123}\v^{-\Lambda_{123}}\xRrightarrow{~~~}\v^{-\Lambda_{123}}\Phi_{123}\quad&&&,\\
&\int_{\v^{\Lambda_{234}}\Phi_{234}}^{\Phi_{234}\v^{\Lambda_{234}}}&&:~\v^{\Lambda_{234}}\Phi_{234}~\,\xRrightarrow{~~~}\Phi_{234}\v^{\Lambda_{234}}&&&,\\
&\int_{\v^{2t_{23}}\Phi_{1(23)4}}^{\Phi_{1(23)4}\v^{2t_{23}}}&&:\v^{2t_{23}}\Phi_{1(23)4}\xRrightarrow{~~~}\Phi_{1(23)4}\v^{2t_{23}}&&&
\end{alignat}
\end{subequations}
are analogous to \eqref{eq:mod commuting Phi_213 with e^lambda} thus we directly compute them as, respectively:
\begin{subequations}
\begin{align}
&\sum_{\begin{smallmatrix}
\{p,q>0\,|\,\p=\q\}\\1\leq m<\infty
\end{smallmatrix}}\sum_{\begin{smallmatrix}0\leq j\leq p\\0\leq k\leq q\end{smallmatrix}}\sum_{\begin{smallmatrix}0\leq l\leq\p+1\\1\leq n\leq m\end{smallmatrix}}\frac{(-\ln\v)^m}{m!}\zeta_{j,k}^{p,q}\Lambda_{123}^{n-1}\\
&\qquad~\times\left(\prod_{r=0}^{l-1}t_{12}^{j_r}t_{23}^{k_r}\right)\left(\sum_{r=1}^{j_l}t_{12}^{r-1}\L_{123}t_{12}^{j_l-r}t_{23}^{k_l}+t_{12}^{j_l}\sum_{r=1}^{k_l}t_{23}^{r-1}\R_{123}t_{23}^{k_l-r}\right)\left(\prod_{r=l+1}^{\p+1}t_{12}^{j_r}t_{23}^{k_r}\right)\Lambda_{123}^{m-n}\quad,\nn
\end{align}
\begin{align}
-&\sum_{\begin{smallmatrix}
\{p,q>0\,|\,\p=\q\}\\1\leq m<\infty
\end{smallmatrix}}\sum_{\begin{smallmatrix}0\leq j\leq p\\0\leq k\leq q\end{smallmatrix}}\sum_{\begin{smallmatrix}0\leq l\leq\p+1\\1\leq n\leq m\end{smallmatrix}}\frac{(\ln\v)^m}{m!}\zeta_{j,k}^{p,q}\Lambda_{234}^{n-1}\\
&\qquad~\times\left(\prod_{r=0}^{l-1}t_{23}^{j_r}t_{34}^{k_r}\right)\left(\sum_{r=1}^{j_l}t_{23}^{r-1}\L_{234}t_{23}^{j_l-r}t_{34}^{k_l}+t_{23}^{j_l}\sum_{r=1}^{k_l}t_{34}^{r-1}\R_{234}t_{34}^{k_l-r}\right)\left(\prod_{r=l+1}^{\p+1}t_{23}^{j_r}t_{34}^{k_r}\right)\Lambda_{234}^{m-n}\nn
\end{align}
and
\begin{align}
&\sum_{\begin{smallmatrix}
\{p,q>0\,|\,\p=\q\}\\1\leq m<\infty
\end{smallmatrix}}\sum_{\begin{smallmatrix}0\leq j\leq p\\0\leq k\leq q\end{smallmatrix}}\sum_{\begin{smallmatrix}0\leq l\leq\p+1\\1\leq n\leq m\end{smallmatrix}}\frac{(2\ln\v)^m}{m!}\zeta_{j,k}^{p,q}t_{23}^{n-1}\left(\prod_{r=0}^{l-1}t_{1(23)}^{j_r}t_{(23)4}^{k_r}\right)\\
&\qquad\qquad\quad\times\left(\sum_{r=1}^{j_l}t_{1(23)}^{r-1}\R_{123}t_{1(23)}^{j_l-r}t_{(23)4}^{k_l}+t_{1(23)}^{j_l}\sum_{r=1}^{k_l}t_{(23)4}^{r-1}\L_{234}t_{(23)4}^{k_l-r}\right)\left(\prod_{r=l+1}^{\p+1}t_{1(23)}^{j_r}t_{(23)4}^{k_r}\right)t_{23}^{m-n}\quad.\nn
\end{align}
\end{subequations}
Composing \eqref{subeq:mods commuting Phi_234, Phi_1(23)4 and Phi_123} with \eqref{eq:2nd move}, we have
\begin{align}
M_3:\v^{\Lambda_{234}}\Phi_{234}\Phi_{1(23)4}\Phi_{123}\v^{-\Lambda_{123}}\Rrightarrow\v^{-t_{34}}\v^{t_{2(34)}}\Phi_{12(34)}\v^{-2t_{12}}\v^{2t_{34}}\Phi_{(12)34}\v^{-t_{(12)3}}\v^{t_{12}}\label{eq:3rd move}
\end{align}
where
\begin{align}
M_3:=~&\v^{\Lambda_{234}}\Phi_{234}\Phi_{1(23)4}\int_{\Phi_{123}\v^{-\Lambda_{123}}}^{\v^{-\Lambda_{123}}\Phi_{123}}\nn\\
&+\Bigg(\int_{\v^{\Lambda_{234}}\Phi_{234}}^{\Phi_{234}\v^{\Lambda_{234}}}\Phi_{1(23)4}+\Phi_{234}\v^{\Lambda_{234}}\v^{-2t_{23}}\int_{\v^{2t_{23}}\Phi_{1(23)4}}^{\Phi_{1(23)4}\v^{2t_{23}}}\Bigg)\v^{-\Lambda_{123}}\Phi_{123}+M_2\quad.
\end{align}
We rearrange \eqref{eq:3rd move} as
\begin{align}
M_4:\Phi_{234}\Phi_{1(23)4}\Phi_{123}\Rrightarrow\v^{-\Lambda_{234}}\v^{-t_{34}}\v^{t_{2(34)}}\Phi_{12(34)}\v^{-2t_{12}}\v^{2t_{34}}\Phi_{(12)34}\v^{-t_{(12)3}}\v^{t_{12}}\v^{\Lambda_{123}}\label{eq:4th move}
\end{align}
where $M_4:=\v^{-\Lambda_{234}}M_3\v^{\Lambda_{123}}$. We have $\v^{-2t_{12}}\v^{2t_{34}}=\v^{2t_{34}}\v^{-2t_{12}}$ hence we consider
\begin{equation}\label{subeq:mods commuting Phi_12(34), Phi_(12)34}
\int_{\v^{-2t_{12}}\Phi_{(12)34}}^{\Phi_{(12)34}\v^{-2t_{12}}}:\v^{-2t_{12}}\Phi_{(12)34}\Rrightarrow\Phi_{(12)34}\v^{-2t_{12}}\quad,\quad\int^{\v^{2t_{34}}\Phi_{12(34)}}_{\Phi_{12(34)}\v^{2t_{34}}}:\Phi_{12(34)}\v^{2t_{34}}\Rrightarrow\v^{2t_{34}}\Phi_{12(34)}
\end{equation}
which are analogous to \eqref{subeq:mods commuting Phi_234, Phi_1(23)4 and Phi_123} thus we directly compute them as, respectively,
\begin{subequations}
\begin{align}
&\sum_{\begin{smallmatrix}
\{p,q>0\,|\,\p=\q\}\\1\leq m<\infty
\end{smallmatrix}}\sum_{\begin{smallmatrix}0\leq j\leq p\\0\leq k\leq q\\0\leq l\leq\p+1\\1\leq n\leq m\end{smallmatrix}}\tfrac{(-2\ln\v)^m}{m!}\zeta_{j,k}^{p,q}t_{12}^{n-1}\left[\prod_{r=0}^{l-1}t_{(12)3}^{j_r}t_{34}^{k_r}\right]\sum_{r=1}^{j_l}t_{(12)3}^{r-1}\L_{123}t_{(12)3}^{j_l-r}t_{34}^{k_l}\left[\prod_{r=l+1}^{\p+1}t_{(12)3}^{j_r}t_{34}^{k_r}\right]t_{12}^{m-n}
\end{align}
and 
\begin{align}
&-\sum_{\begin{smallmatrix}
\{p,q>0\,|\,\p=\q\}\\1\leq m<\infty
\end{smallmatrix}}\sum_{\begin{smallmatrix}0\leq j\leq p\\0\leq k\leq q\\0\leq l\leq\p+1\\1\leq n\leq m\end{smallmatrix}}\tfrac{(2\ln\v)^m}{m!}\zeta_{j,k}^{p,q}t_{34}^{n-1}\left[\prod_{r=0}^{l-1}t_{12}^{j_r}t_{2(34)}^{k_r}\right]t_{12}^{j_l}\sum_{r=1}^{k_l}t_{2(34)}^{r-1}\R_{234}t_{2(34)}^{k_l-r}\left[\prod_{r=l+1}^{\p+1}t_{12}^{j_r}t_{2(34)}^{k_r}\right]t_{34}^{m-n}
\end{align}
\end{subequations}
which we can compose with \eqref{eq:4th move} to give
\begin{align}\label{eq:5th move}
M_5:\Phi_{234}\Phi_{1(23)4}\Phi_{123}\Rrightarrow\v^{-\Lambda_{234}}\v^{-t_{34}}\v^{t_{2(34)}}\v^{2t_{34}}\Phi_{12(34)}\Phi_{(12)34}\v^{-2t_{12}}\v^{-t_{(12)3}}\v^{t_{12}}\v^{\Lambda_{123}}
\end{align}
where 
\begin{align}
M_5:=~&\v^{-\Lambda_{234}}\v^{-t_{34}}\v^{t_{2(34)}}\left(\Phi_{12(34)}\v^{2t_{34}}\int_{\v^{-2t_{12}}\Phi_{(12)34}}^{\Phi_{(12)34}\v^{-2t_{12}}}+\int^{\v^{2t_{34}}\Phi_{12(34)}}_{\Phi_{12(34)}\v^{2t_{34}}}\Phi_{(12)34}\v^{-2t_{12}}\right)\v^{-t_{(12)3}}\v^{t_{12}}\v^{\Lambda_{123}}\nn\\
&+M_4
\end{align}
Penultimately, we consider
\begin{subequations}\label{subeq:R_1 and L_4}
\begin{align}
\int_{\v^{-t_{12}}\v^{-t_{(12)3}}}^{\v^{-t_{(12)3}}\v^{-t_{12}}}=\sum_{\begin{smallmatrix}j=1\\k=1
\end{smallmatrix}}^\infty\frac{(-\ln\v)^{j+k}}{j!k!}\sum_{\begin{smallmatrix}1\leq l\leq j\\1\leq m\leq k
\end{smallmatrix}}t_{(12)3}^{m-1}t_{12}^{l-1}\L_{123}t_{12}^{j-l}t_{(12)3}^{k-m}:\v^{-t_{12}}\v^{-t_{(12)3}}\Rrightarrow\v^{-t_{(12)3}}\v^{-t_{12}}
\end{align}
and
\begin{align}
\int^{\v^{t_{2(34)}}\v^{-t_{34}}}_{\v^{-t_{34}}\v^{t_{2(34)}}}=\sum_{\begin{smallmatrix}j=1\\k=1
\end{smallmatrix}}^\infty(-1)^j\frac{(\ln\v)^{j+k}}{j!k!}\sum_{\begin{smallmatrix}1\leq l\leq j\\1\leq m\leq k
\end{smallmatrix}}t_{2(34)}^{m-1}t_{34}^{l-1}\R_{234}t_{34}^{j-l}t_{2(34)}^{k-m}:\v^{-t_{34}}\v^{t_{2(34)}}\Rrightarrow\v^{t_{2(34)}}\v^{-t_{34}}
\end{align}
\end{subequations}
which we compose with \eqref{eq:5th move} to give
\begin{align}\label{eq:6th move}
M_6:\Phi_{234}\Phi_{1(23)4}\Phi_{123}\Rrightarrow\v^{-\Lambda_{234}}\v^{t_{2(34)}}\v^{t_{34}}\Phi_{12(34)}\Phi_{(12)34}\v^{-t_{12}}\v^{-t_{(12)3}}\v^{\Lambda_{123}}
\end{align}
where
\begin{align}
M_6~:=~~&M_5+\v^{-\Lambda_{234}}\int^{\v^{t_{2(34)}}\v^{-t_{34}}}_{\v^{-t_{34}}\v^{t_{2(34)}}}\v^{2t_{34}}\Phi_{12(34)}\Phi_{(12)34}\v^{-2t_{12}}\v^{-t_{(12)3}}\v^{\Lambda_{123}}\nn\\&+\v^{-\Lambda_{234}}\v^{t_{2(34)}}\v^{t_{34}}\Phi_{12(34)}\Phi_{(12)34}\v^{-t_{12}}\int_{\v^{-t_{12}}\v^{-t_{(12)3}}}^{\v^{-t_{(12)3}}\v^{-t_{12}}}\v^{t_{12}}\v^{\Lambda_{123}}\quad.
\end{align}
Lastly, we consider the modifications 
\begin{equation}\label{subeq:BCH mods for (12)3 and 2(34)}
\int_{\v^{\Lambda_{123}}}^{\v^{t_{(12)3}}\v^{t_{12}}}:\v^{\Lambda_{123}}\xRrightarrow{~~~}\v^{t_{(12)3}}\v^{t_{12}}\qquad,\qquad\int_{\v^{-\Lambda_{234}}}^{\v^{-t_{34}}\v^{-t_{2(34)}}}:\v^{-\Lambda_{234}}\xRrightarrow{~~~}\v^{-t_{34}}\v^{-t_{2(34)}}
\end{equation}
which are analogous to \eqref{subeq:formula for BCH mods for (23)4 and 1(23)} and are given as, respectively,
\begin{subequations}\label{subeq:formula for BCH mods for (12)3 and 2(34)}
\begin{alignat}{6}
&\sum_{k=2}^\infty\frac{(\ln\v)^k}{k!}\sum_{l=1}^{k-1}\sum_{m=0}^{k-l-1}\sum_{n=0}^{k-l-m-1}\binom{k-l}{m}\Lambda_{123}^{l-1}t_{(12)3}^n\L_{123}t_{(12)3}^{k-l-m-n-1}t_{12}^m&&,\\
&\sum_{k=2}^\infty\frac{(\ln\v)^k}{k!}\sum_{l=1}^{k-1}\sum_{m=0}^{k-l-1}\sum_{n=0}^{k-l-m-1}\binom{k-l}{m}(-1)^{k-1}\Lambda_{234}^{l-1}t_{34}^n\R_{234}t_{34}^{k-l-m-n-1}t_{2(34)}^m\quad&&.
\end{alignat}
\end{subequations}
Finally, we compose \eqref{eq:6th move} with \eqref{subeq:BCH mods for (12)3 and 2(34)} to arrive at
\begin{align}
\Pi~:=~~&M_6+\v^{-\Lambda_{234}}\int_{\v^{-\Lambda_{234}}}^{\v^{-t_{34}}\v^{-t_{2(34)}}}\v^{t_{2(34)}}\v^{t_{34}}\Phi_{12(34)}\Phi_{(12)34}\v^{-t_{12}}\v^{-t_{(12)3}}\v^{\Lambda_{123}}\nn\\
&+\Phi_{12(34)}\Phi_{(12)34}\v^{-t_{12}}\v^{-t_{(12)3}}\int_{\v^{\Lambda_{123}}}^{\v^{t_{(12)3}}\v^{t_{12}}}\quad.
\end{align}
\end{proof}
%%%%%%%%%%%%%%%%%%%%%%%%%%%%%%%%%%%%%%%%%%%%%%%%%%%%%%

\end{document}